\documentclass[smallextended]{svjour3}
\smartqed

\usepackage{mathptmx}
\usepackage{amsfonts}
\usepackage{dsfont}
\usepackage{graphicx}
\usepackage{subfig}
\usepackage{multirow}
\usepackage{epstopdf}

\newcommand{\cF}{\mathcal{F}}
\newcommand{\ip}[2]{\ensuremath{\langle #1,#2\rangle}}
\newcommand{\Rset}{\mathbb{R}}
\newcommand{\ri}{\mbox{ri}}
\newcommand{\cone}{\mbox{cone}}
\newtheorem{assumption}{Assumption}

\journalname{}

\begin{document}

\author{ Michael Lamm \and Shu Lu \and Amarjit Budhiraja.
}
\title{Individual confidence intervals for true solutions to stochastic variational inequalities}

\institute{ M. Lamm \at Department of Statistics and Operations Research, University of North Carolina at Chapel Hill, B05 Hanes Hall, CB\#3260, Chapel Hill, NC 27599-3260 \\ \email{mlamm@email.unc.edu}
\and S. Lu \at Department of Statistics and Operations Research, University of North Carolina at Chapel Hill, 355 Hanes Hall, CB\#3260, Chapel Hill, NC 27599-3260\\  \email{shulu@email.unc.edu}
\and A. Budhiraja \at Department of Statistics and Operations Research, University of North Carolina at Chapel Hill, 357 Hanes Hall, CB\#3260, Chapel Hill, NC 27599-3260 \\ \email{budhiraja@email.unc.edu}}.

\date{Received: date / Accepted: date}

\maketitle
\begin{abstract}
Stochastic variational inequalities (SVI) provide a means for modeling various optimization and equilibrium problems where data are subject to uncertainty. Often it is necessary to estimate the true SVI solution by the solution of a sample average approximation (SAA) problem. This paper proposes three methods for building confidence intervals for components of the true solution, and those intervals are computable from a single SAA solution. The first two methods use an ``indirect approach'' that requires initially computing asymptotically exact confidence intervals for the solution to the normal map formulation of the SVI. 
The third method directly constructs confidence intervals for the true SVI solution; intervals produced with this method meet a minimum specified level of confidence in the same situations for which the first two methods are applicable. We justify the three methods theoretically with weak convergence results, discuss how to implement these methods, and test their performance using two numerical examples.
\end{abstract}
\section{Introduction}\label{s:intro}

This paper considers the problem of building individual confidence intervals for components of the true solution to a stochastic variational inequality (SVI). An SVI is defined as follows. Let \( (\Omega, \cF, P\)) be a probability space, and $\xi$ be a $d$-dimensional random vector defined on $\Omega$ and supported on a closed subset \( \Xi\) of \(\Rset^d\). Let \(O\) be an open subset of \(\Rset^n\), and \(F\) be a measurable function from \(O\times \Xi\) to \(\Rset^n\), such that for each \(x\in O\), \(E\|F(x,\xi)\| < \infty\). Let \(S\) be a polyhedral convex set in \(\Rset^n\). The SVI problem is to find a point \(x\in S\cap O\) such that
\begin{equation}\label{q:true}
 0\in f_0(x) + N_S(x),
\end{equation}
where \(f_0(x)=E \left[ F(x,\xi)\right] \) and \(N_S(x)\subset \Rset^n\) denotes the normal cone to \(S\) at \(x\):
\[
N_S(x)=\left\{v\in \Rset^n | \langle v, s-x \rangle \le 0  \mbox{ for each } s\in S\right\}.
\]
Here \(\langle\cdot , \cdot \rangle\) denotes the scalar product of two vectors of the same dimension.

Variational inequalities provide a means for modeling a variety of optimization and equilibrium problems, see \cite[Chapter 1]{faccpang:fdvi}. Stochastic variational inequalities allow for the incorporation of uncertainty in the model data. As an expectation function, \(f_0\) often does not have a closed form expression and is difficult to evaluate. In such circumstances the problem (\ref{q:true}) is replaced by a suitable approximation. This paper considers the case when a sample average approximation (SAA) is used. The SAA method takes independent and identically distributed (i.i.d) random vectors \( \xi^1,\xi^2,\dots,\xi^N\) with the same distribution as \(\xi\) and constructs a sample average function \(f_N:O\times \Omega \to \Rset^n \) as
\begin{equation}\label{q:saa_fun}
f_N(x,\omega)=N^{-1}\sum_{i=1}^{N} F(x,\xi^i(\omega)).
\end{equation}
The SAA problem is to find for given $\omega\in \Omega$ a point \(x\in O \cap S \) such that
\begin{equation}\label{q:saa}
0\in f_N(x,\omega) + N_S(x).
\end{equation}
We will use $x_0$ to denote a solution to (\ref{q:true}) and refer to it as the true solution, and use $x_N$ to denote a solution to (\ref{q:saa}) and call it an SAA solution; the formal definitions of $x_0$ and $x_N$ will be given in Assumption \ref{assu2} and Theorem \ref{t:asy_dis} respectively.

A natural question to ask is how well the SAA solutions approximate the true solution. An answer to this question depends on the convergence behavior of SAA solutions.  Under certain regularity conditions, SAA solutions are known to converge almost surely to a true solution as the sample size \(N\) goes to infinity, see G\"{u}rkan, \"{O}zge and Robinson \cite{gur.oze.rob:asc}, King and Rockafellar \cite{king.rock:conv}, and Shapiro, Dentcheva and Ruszczy\'{n}ski \cite[Section 5.2.1]{shap.dent.rus:lsp}. Xu \cite{xu:exc} showed the convergence of SAA solutions to the set of true solutions in probability at an exponential rate under  some assumptions on the moment generating functions of certain random variables; related results on the exponential convergence rate are given in \cite{sha.xu:smp}. Working with the exponential rate of convergence of SAA solutions, Anitescu and Petra in \cite{solcibs} developed confidence intervals for the optimal value of stochastic programming problems using bootstrapping. The asymptotic distribution of SAA solutions was obtained in King and Rockafellar \cite[Theorem 2.7]{king.rock:conv} and Shapiro, Dentcheva and Ruszczy\'{n}ski \cite[Section 5.2.2]{shap.dent.rus:lsp}. For random approximations to deterministic optimization problems, universal confidence sets for the true solution set were developed by Vogel in \cite{v:csets} using concentration of measure results.

The objective of this paper is to provide methods to compute confidence intervals for each individual component of the true solution $x_0$ from a single SAA solution $x_N$, based on the asymptotic distribution of SAA solutions. To our knowledge, this line of work started from the dissertation \cite{dem:acr} of Demir. By considering the \emph{normal map formulation} (to be defined formally in \S \ref{s:back}) of variational inequalities, Demir used the asymptotic distribution to obtain an expression for confidence regions of the solution to the normal map formulation of (\ref{q:true}), which we denote by $z_0$ (the formal definition of $z_0$ is in Assumption \ref{assu2}). Because some quantities in that expression depend on the true solutions $x_0$ and $z_0$ and are not computable, Demir proposed a substitution method to make that expression computable. He did not, however, justify why that substitution method preserves the weak convergence property needed for the asymptotic exactness of the confidence regions. The general nonsmooth structure of $S$ creates issues related to   discontinuity of certain quantities, which prevents standard techniques from being applicable for the required justification.

In \cite{l.b:acr} Lu and Budhiraja continued to consider the normal map formulations of both (\ref{q:true}) and (\ref{q:saa}). They provided and justified a new method of constructing asymptotically exact confidence regions for $z_0$, computable from a solution to the normal map formulation of a single SAA problem (\ref{q:saa}); the latter solution is denoted by $z_N$ and is formally defined in Theorem \ref{t:asy_dis}. The approach in \cite{l.b:acr} was to combine the asymptotic distribution of $z_N$ with its exponential rate of convergence, and its computation involved calculating a weighted-sum of a family of functions. The method was later simplified by Lu in \cite{l:acr} by using a single function from the family. When $z_N$ does not asymptotically follow a normal distribution, confidence regions generated from \cite{l:acr} and \cite{l.b:acr} are fractions of multiple ellipses pieced together.
  Lu \cite{lu:crc} proposed a different method to construct asymptotically exact confidence regions, by using only the asymptotic distribution and not the exponential convergence rate. The method in \cite{lu:crc} has the advantage that the confidence region generated from it is with high probability a single ellipse, even when the asymptotic distribution of $z_N$ is not normal, and is therefore easier to use. Nonetheless, methods in \cite{l:acr,l.b:acr} provide valuable information beyond confidence regions. In the present paper we will show how to use such information to compute individual confidence intervals for the true solutions. Even with the estimators from  \cite{l:acr,l.b:acr} in place, it is not straightforward to obtain asymptotically exact individual confidence intervals, due to the piecewise linear structure that underlies the asymptotic distributions of $z_N$ and $x_N$. How to reduce the computational burden related to that piecewise linear structure is another challenge. Those difficulties are what we aim to overcome in this paper.

 Compared to confidence regions, component-wise confidence intervals are usually more convenient to visualize and interpret. By finding the axis-aligned minimal bounding box of a confidence region of $z_0$ (or $x_0$), one can find simultaneous confidence intervals, that jointly contain $z_0$ (or $x_0$) with a probability no less than a prescribed confidence level. However, individual confidence intervals that can be obtained by using confidence regions are too conservative for any practical use, especially for large scale problems. Individual confidence intervals provide a quantitative measure of the uncertainty of each individual component, and therefore carry important information not covered by simultaneous intervals. Lu \cite{lu:crc} proposed a method to construct individual confidence intervals for $z_0$, but that method relies on some restrictive assumptions to guarantee the specified level of confidence is met. The methods we develop in this paper are shown to achieve the guaranteed confidence levels in more general situations.


As noted above, the confidence region/interval methods in \cite{dem:acr,l:acr,lu:crc,l.b:acr} are mainly designed for $z_0$. The points $z_0$ and $x_0$ are related by the equality $x_0 = \Pi_S(z_0)$. From a confidence region of $z_0$, one can obtain a confidence region for $x_0$, by projecting the confidence region of $z_0$ onto $S$. The resulting set will cover $x_0$ with a rate at least as large as the coverage rate of the original confidence region for $z_0$. Simultaneous confidence intervals of $x_0$ can then be obtained from the minimum bounding box of its confidence region. When $S$ is a box, individual confidence intervals of $x_0$ can also be obtained from projecting the individual confidence intervals of $z_0$. We shall refer to such approaches as ``indirect approaches.'' The indirect approaches are convenient to implement when the set \(S\) is a box, or has a similar structure that facilitates taking (individual) projections. Beyond those situations, it would be hard to use the indirect approaches for finding confidence intervals for $x_0$. Another contribution of the present paper is to provide a \emph{direct approach} to finding confidence intervals for $x_0$.

Altogether, this paper presents three new methods for constructing individual confidence intervals, justifies them with weak convergence results, discusses how to implement these methods, and provides numerical examples. The first two methods belong to the aforementioned indirect approaches. They produce confidence intervals for $z_0$ from a single $z_N$, and the asymptotic level of confidence can be specified under general situations. The third method is a direct approach that produces individual confidence intervals for $x_0$. The intervals produced by the third method meet a specified minimum level of confidence in the same situations for which the first two methods are applicable. 
 While our main interest in this paper is on stochastic variational inequalities and their normal map formulations, the ideas of the first two methods work for general piecewise linear homeomorphisms.
We outline the ideas below, and leave formal definitions and proofs to Sections \ref{s:back} and \ref{s:theory}. Throughout, we use \(\mathcal{N}(\nu, \Sigma)\) to denote a Normal random vector with mean \(\nu \) and covariance matrix \(\Sigma\), and use  \(Y_n\Rightarrow Y\) to denote the weak convergence of random variables $Y_n$ to $Y$. For a vector \(v \in \mathbb{R}^n\), \( (v)_j\) will denote the \(j^{\tiny\mbox{th}}\) coordinate. Similarly for a function \(f:\mathbb{R}^n \rightarrow \mathbb{R}^n\), \( (f)_j\) will denote the \(j^{\tiny\mbox{th}}\) component function. We use \( \| \cdot \| \) to denote the norm of an element in a normed space; unless a specific norm is stated it can be any norm, as long as the same norm is used in all related contexts.

For the first two methods, suppose \(f:\Rset^n \rightarrow \Rset^n\) is a piecewise linear homeomorphism with a family of selection functions \( \left\{ M_1,\dots,M_l \right\}\) and the corresponding conical subdivision \(\left\{ K_1,\dots,K_l\right\}\), so $f$ is represented by the linear map $M_i$ when restricted to $K_i$. Suppose $z_N$ is an $n$-dimensional random vector such that $\sqrt{N} (z_N-z_0)\Rightarrow f^{-1} (Z)$, where $z_0\in\Rset^n$ is an unknown parameter, \( Z \sim \mathcal{N}(0,I_n)\), and $I_n$ is the $n\times n$ identity matrix.  Our objective is to obtain a confidence interval for $(z_0)_j$, $j=1,\cdots,n$. The idea of the first method is to look for a number $a$ such that $\Pr(|(f^{-1})_j(Z)|\le a)$ equals a prescribed confidence level, and then use $[(z_N)_j -a{N}^{-1/2}, (z_N)_j +a{N}^{-1/2}]$ as the interval. For situations considered in this paper, $z_0$ and $z_N$ are solutions to the normal map formulations of (\ref{q:true}) and (\ref{q:saa}) respectively, and the function $f$ is unknown and is substituted by an estimator obtained from approaches in \cite{l:acr} and \cite{l.b:acr}. Such a substitution does not affect the asymptotic exactness of confidence intervals computed from this method, as we show in Theorem \ref{t:mainres}. In addition, to allow for some choice in where the interval is centered, we introduce a parameter $r$ and consider the probability $\Pr(|(f^{-1})_j(Z)-r|\le a)$.

A challenge that arises with the first method is that when the function $f$ is piecewise linear we lack a closed form expression for the value of $a$. The computation of $\Pr(|(f^{-1})_j(Z)-r|\le a)$ for fixed $a$ and $r$ requires enumerating all pieces of $f^{-1}$, and for each such piece one needs to compute the probability for some normal random vector to belong to a certain polyhedron. Thus, the  calculations necessary to find a confidence interval increase with the number of pieces in $f$. These limitations lead to the consideration of upper bounds for interval half-widths, presented in \S  \ref{s:comp}, and the development of the second method in this paper.

The second method uses the idea of conditioning. Suppose that for each $\omega\in \Omega$ we can identify a cone $K(\omega)$, such that with high probability $K(\omega)$ belongs to the family \(\left\{ K_1,\dots,K_l\right\}\) and contains $z_N - z_0$ in its interior; for situations in this paper this can be done using an approach in \cite{l:acr}. For the fixed $\omega$ we then look for a number $a(\omega)$ such that the following conditional probability
\[\frac{\Pr\left(|(f^{-1})_j(Z)|\le a(\omega),\; f^{-1}Z\in K(\omega)\right)}{\Pr\left(f^{-1}Z\in K(\omega)\right)}\]
equals a prescribed confidence level, and use $[(z_N)_j -a(\omega){N}^{-1/2}, (z_N)_j +a(\omega){N}^{-1/2}]$ as a confidence interval for $(z_0)_j$. 
We will again use an estimator to replace the unknown $f$, and justify the method with a convergence result (Theorem \ref{t:meth2}). The second method dramatically reduces the computation needed for the first method, by focusing on a single cone $K(\omega)$ and avoiding the enumeration of all pieces of $f$.

The third method also uses the idea of conditioning, but it is a direct approach and is different from the second method. In general, one cannot apply the first two methods or the method in \cite{lu:crc} directly to compute individual confidence intervals for $x_0$, because if one would put the asymptotic distribution of $x_N$ in the form $\sqrt{N}(x_N - x_0) \Rightarrow f(Z)$ for some function $f$ then $f$ is generally non-invertible. Such non-invertibility relates to a fact that there is possibly a nonzero probability for some components of $x_N$ and $x_0$ to coincide, a situation that does not occur when considering solutions to the normal map formulations. The third method handles that non-invertibility by looking into the exact cause of such non-invertibility, and produces intervals that meet a minimum specified level of confidence in the same situations for which the first two methods are shown to be asymptotically exact. In the proof of the convergence result for this method (Theorem \ref{t:icim}), we see that the intervals it produces exceed the specified level of confidence only if the corresponding components of $x_N$ and $x_0$ have a nonzero probability to coincide. When the latter situation happens, the third method returns a point estimate with a nonzero probability.

The organization of this paper is as follows. Section \ref{s:back} reviews pertinent background material on piecewise linear functions, the normal map formulation and previous asymptotics results. The main theoretical results of this paper are presented in \S \ref{s:theory}, and \S  \ref{s:comp} presents computational approaches for calculating intervals using these results. The paper concludes in \S \ref{s:examps} with two numerical examples.

\section{Background}\label{s:back}

In this section we discuss the normal map formulation of a variational inequality, pertinent properties of piecewise linear functions, the notion of B-differentiability and previous works on this topic.

 For \(f_0\) and \( S\) defined as above the \emph{normal map} induced by \(f_0\) and \(S\) is a function \( (f_0)_S: \Pi_S^{-1}(O) \rightarrow \mathbb{R}^n\), defined as
\begin{equation}\label{q:nmtrue}
(f_0)_S(z)=f_0(\Pi_S(z))+(z-\Pi_S(z)).
\end{equation}
Here \( \Pi_S \) denotes the Euclidian projector onto the set \(S\) and \( \Pi_S^{-1}(O) \) is the set of all points \( z \in \mathbb{R}^n \) such that \( \Pi_S(z) \in O\). One can check that \(x \in S \cap O \) is a solution to (\ref{q:true}) only if \(z=x-f_0(x)\) satisfies
\begin{equation}\label{q:nm_tform}
(f_0)_S(z)=0.
\end{equation}
When the above equality is satisfied, one also has \( \Pi_S(z)=x\). We refer to (\ref{q:nm_tform}) as the normal map formulation of (\ref{q:true}).

The normal map induced by \(f_N\) and \(S\) is similarly defined to be a function on \( \Pi_S^{-1}(O)\):
\begin{equation}\label{q:nmsaa}
(f_N)_S(z)=f_N(\Pi_S(z))+(z-\Pi_S(z)).
\end{equation}
The normal map formulation of the SAA problem (\ref{q:saa}) is then
\begin{equation}\label{q:nm_sform}
(f_N)_S(z)=0,
\end{equation}
where (\ref{q:nm_sform}) and (\ref{q:saa}) are related in the same manner as (\ref{q:nm_tform}) and (\ref{q:true}). In general for a function \(G\) mapping from a subset \(D\) of \(\mathbb{R}^n\) back into \(\mathbb{R}^n\), the normal map $G_S$ induced by \(G\) and \(S\) is a map from \(\Pi^{-1}_S(D)\) into \(\mathbb{R}^n\) with \(G_S(z)=G(\Pi_S(z))+z-\Pi_S(z)\).

Since \(S\) is a polyhedral convex set, the Euclidian projector \(\Pi_S\) is a piecewise affine function. A continuous function \(f:\mathbb{R}^n \rightarrow \mathbb{R}^k\) is piecewise affine if there exists a finite family of affine functions \(f_j: \mathbb{R}^n \rightarrow \mathbb{R}^k\), \(j=1,\dots,m\), such that for all \(x \in \mathbb{R}^n\) \(f(x) \in \left\{ f_1(x),\dots,f_m(x) \right\} \). The affine functions \(f_i\), \(i=1,\dots,m\), are referred to as the selection functions of \(f\). When each selection function is linear the function \(f\) is called piecewise linear.

Closely related to piecewise affine functions is the concept of a polyhedral subdivision. A polyhedral subdivision of \( \mathbb{R}^n\) is defined to be a finite collection of convex polyhedra, \( \Gamma=\{ \gamma_1, \dots, \gamma_m\} \subset \mathbb{R}^n \), satisfying the following three conditions:
\begin{enumerate}
\item Each \( \gamma_i\) is of dimension \(n\).
\item The union of all the \(\gamma_i\) is \ \( \mathbb{R}^n\).
\item The intersection of any two \( \gamma_i \) and \( \gamma_j \), \( 1 \le i \ne  j \le m\), is either empty or a common proper face of both \( \gamma_i\) and \( \gamma_j\).
\end{enumerate}
If each of the \( \gamma_i\) is additionally a cone, then \( \Gamma \) is referred to as a conical subdivision. As seen in \cite[Proposition 2.2.3]{sch:ipdf}, for every piecewise affine function \( f \) there is a corresponding polyhedral subdivision of \( \mathbb{R}^n\) such that the restriction of \( f \) to each \( \gamma_i\) is an affine function. When \( f\) is piecewise linear the corresponding subdivision is conical, and the restriction of \(f \) to each cone of the subdivision a linear function.

We now return to the special case of the Euclidian projector onto a polyhedral convex set \(S \subset \mathbb{R}^n\), a thorough discussion of which can be found in \cite[Section 2.4]{sch:ipdf}.  Let \( \mathcal{F}\) be the finite collection of all nonempty faces of \(S\). On the relative interior of each nonempty face \(F\in \mathcal{F}\) the normal cone to \(S\) is a constant cone, denoted as \(N_S(\ri F)\), and \(C_F=F+N_S(\ri F)\) is a polyhedral convex set of dimension \(n\). The collection of all such sets \(C_F\) form the polyhedral subdivision of \( \mathbb{R}^n\) corresponding to \(\Pi_S\). This collection of sets is also referred to as the normal manifold of \(S\), with each \(C_F\) called an \(n\)-cell in the normal manifold. Each \(k\)-dimensional face of an \(n\)-cell is called a \(k\)-cell in the normal manifold for \(k=0,1,\dots,n\). The relative interiors of all cells in the normal manifold of \(S\) form a partition of \(\mathbb{R}^n\).

Next we introduce the concept of B-differentiability. A function \(h:\mathbb{R}^n \rightarrow \mathbb{R}^m\) is said to be B-differentiable at a point \(x \in \mathbb{R}^n\) if there exists a positive homogeneous function, \(H:\mathbb{R}^n \rightarrow \mathbb{R}^m\), such that
\[
h(x+v)=h(x)+H(v)+o(v).
\]
Recall that a function \(G\) is positive homogeneous if for all positive numbers \( \lambda \in \mathbb{R} \) and points \( x\in \mathbb{R}^n\) \(G(\lambda x)=\lambda G(x)\). The function \(H\) is referred to as the B-derivative of \(h\) at $x$ and will be denoted \(dh(x)\). When in addition to \(dh(x)\) being positive homogeneous it is also linear, \(dh(x)\) is the classic Fr\'{e}chet derivative (F-derivative). A function \(h: U \times V \to Z\) is partially B-differentiable in \(x\) at \((x_0,y_0)\in U\times V\), if the function \(h(\cdot,y_0)\) is B-differentiable at \(x_0\). The partial B-derivative is denoted by \(d_xh(x_0,y_0)\).

A piecewise affine function \(f\), while not F-differentiable at all points, is B -differentiable everywhere. More precisely, let \(\Gamma\) be the polyhedral subdivision associated with \(f\). At points \(x\) in the interior of a polyhedra \(\gamma_i \in \Gamma\), \(df(x)\) is a linear function equal to \(df_i(x)\), the F-derivative of the corresponding selection function $f_i$. When \(x\) lies in the intersection of two or more polyhedra let  \( \Gamma (x)=\left\{ \gamma_i \in \Gamma | x \in \gamma_i \right\}\), \(I=\left\{i | \gamma_i \in \Gamma(x) \right\}\) and \(\Gamma'(x)=\left\{ \cone(\gamma_i-x) | i \in I  \right\} \). That is, $\Gamma (x)$ is the collection of elements in $\Gamma$ that contain $x$, and $\Gamma'(x)$ is the ``globalization'' of $\Gamma(x)$ along with a shift of the origin. With this notation, \(df(x)\) is piecewise linear with a family of selection functions \(\left\{ df_i(x) | i \in I\right\} \) and the corresponding conical subdivision \( \Gamma'(x)\).

The following four assumptions are used to prove pertinent asymptotic properties of SAA solutions.

\begin{assumption}\label{assu1}
(a)   \(E\|F( x,\xi)\|^2 < \infty\) for all \(x \in O\).\\
\noindent (b) The map \(x \mapsto F(x, \xi(\omega))\) is continuously differentiable on \(O\) for a.e. \(\omega \in \Omega\).\\
\noindent (c) There exists a square integrable random variable \(C\) such that for all \(x,x'\in O\)
\[
\|F(x, \xi(\omega)) - F(x', \xi(\omega))\| + \|d_xF(x, \xi(\omega)) - d_xF(x', \xi(\omega))\| \le C(\omega) \|x-x'\|,
\]
for a.e. \(\omega \in \Omega\).
\end{assumption}

From Assumption \ref{assu1} it follows that \(f_0\) is continuously differentiable on \( O\), see, e.g., \cite[Theorem 7.44]{shap.dent.rus:lsp}. For any nonempty compact subset \(X\) of \(O\), let \(C^{1} (X, \Rset^n)\) be the Banach space of continuously differentiable mappings \(f:X \to \Rset^n\), equipped with the norm
\begin{equation}\label{q:norm_C1}
\|f\|_{1,X} = \sup_{x\in X} \|f(x)\| + \sup_{x\in X} \|d f(x)\|.
\end{equation}
Then in addition to providing nice integrability properties for \(f_N\), as shown in \cite[Theorem 7.48]{shap.dent.rus:lsp} Assumption 1 will guarantee the almost sure convergence of the sample average function \(f_N\) to \(f_0\) as an element of \(C^{1} (X, \Rset^n)\) and that \(df_0(x)=E\left[ d_xF(x,\xi) \right]\).

Before stating the second assumption we must first define two sets related to the geometric structure of the set \(S\). For polyhedral convex \(S\), the tangent cone to \(S\) at a point \(x \in S\) is defined to be
\[
T_S(x)= \{v\in \Rset^n | \mbox{ there exists } t > 0  \mbox{ such that } x + tv \in S\},
\]
and the critical cone to \(S\) at a point \(z \in \mathbb{R}^n\) is
\[
K(z)=T_S(\Pi_S(z)) \cap \{z-\Pi_S(z)\}^\perp.
\]

\begin{assumption}\label{assu2}
Suppose that \(x_0\) solves the variational inequality (\ref{q:true}). Let \(z_0=x_0-f_0(x_0)\), \(L=d f_0(x_0)\), \(K_0=T_S(x_0) \cap \{z_0-x_0\}^\perp\), and assume that the normal map  \(L_{K_0}\) induced by \(L\) and \(K_0\) is a homeomorphism from \(\Rset^n\) to \(\Rset^n\).
\end{assumption}

\begin{assumption}\label{assu3}
Let \( \Sigma_0\) denote the covariance matrix of \(F(x_0,\xi)\). Suppose that the determinant of \(\Sigma_0\) is strictly positive.
\end{assumption}

\begin{assumption}\label{assu4}
(a) For each \(t \in \mathbb{R}^n\) and \(x \in X\), let
\[
M_x(t)=E\left[ \exp \left\{ \ip{t}{F(x,\xi)-f_0(x)} \right\} \right]
\]
be the moment generating function of the random variable \(F(x,\xi)-f_0(x)\). Assume
\begin{enumerate}
\item There exists \( \zeta > 0\) such that \( M_x(t) \le \exp \left\{ \zeta^2\|t\|^2/2\right\} \) for every \(x \in X\) and every \( t \in \mathbb{R}^n\).
\item There exists a nonnegative random variable \( \kappa\) such that
\[
\|F(x,\xi(\omega))-F(x',\xi(\omega))\| \le \kappa(\omega) \|x-x'\|
\]
for all \(x,x' \in O\) and almost every \(\omega \in \Omega\).
\item The moment generating function of \(\kappa\) is finite valued in a neighborhood of zero.\\
\end{enumerate}
(b) For each \(T\in\mathbb{R}^{n \times n}\) and \(x \in X\), let
\[
\mathcal{M}_x(T)=E\left[ \exp \left\{ \ip{T}{d_xF(x,\xi)-df_0(x)} \right\} \right]
\]
be the moment generating function of the random variable \(d_xF(x,\xi)-df_0(x)\). Assume
\begin{enumerate}
\item There exists \( \varsigma >0\) such that \(\mathcal{M}_x(T) \le \exp \left\{ \varsigma^2 \|T\|^2/2 \right\}\) for every \(x \in X\) and every \( T \in \mathbb{R}^{n \times n}\).
\item There exists a nonnegative random variable \( \nu\) such that
\[
\|d_xF(x,\xi(\omega))-d_xF(x',\xi(\omega))\| \le \nu(\omega) \|x-x'\|
\]
for all \(x,x' \in O\) and almost every \( \omega \in \Omega\).
\item The moment generating function of \( \nu \) is finite valued in a neighborhood of zero.
\end{enumerate}
\end{assumption}

Assumptions \ref{assu1} and \ref{assu2} ensure that the variational inequality (\ref{q:true}) has a locally unique solution under sufficiently small perturbations of \(f_0\) in \(C^1(X,\mathbb{R}^n)\), see \cite[Lemma 1]{l.b:acr} and the original result in \cite{smr:sav}. It is worth mentioning the relation between the normal map $L_{K_0}$ in Assumption 2 and the normal map $(f_0)_S$. As a piecewise affine function, $\Pi_S$ is B-differentiable. If we use $\Gamma$ to denote the normal manifold of $S$, then $\Gamma$ is also the polyhedral subdivision for $\Pi_S$. Following the discussion of B-differentiability above Assumption 1, $\Gamma'(z_0)$ denotes the conical subdivision that corresponds to  \( d\Pi_S(z_0)\). Since \(f_0\) is differentiable from Assumption 1, the chain rule of B-differentiability implies \( (f_0)_S \) to be B-differentiable, with its B-derivative at \(z_0\) given by
\begin{equation}\label{q:bdivtrue}
d(f_0)_S(z_0)(h)=df_0(x_0)(d\Pi_S(z_0)(h))+h-d\Pi_S(z_0)(h).
\end{equation}
 The conical subdivision for \( d(f_0)_S(z_0)\) is also $\Gamma'(z_0)$. Moreover, as shown in \cite[Corollary 4.5]{rob:ift} and \cite[Lemma 5]{jsp:kcone}, for any point \(z \in \mathbb{R}^n\) and \(h \in\mathbb{R}^n\) sufficiently small the equality
\begin{equation}\label{q:SnK}
\Pi_S(z+h)=\Pi_S(z)+\Pi_{K(z)}(h)
\end{equation}
holds, which implies
\begin{equation}\label{q:dSnK}
d\Pi_S(z)=\Pi_{K(z)} \mbox{ for any }z \in \Rset^n.
\end{equation}
Applying (\ref{q:dSnK}) to $z_0$, one can see the normal map \(L_{K_0}\) is exactly \(d(f_0)_S(z_0)\), a result that first appeared in \cite{rob:nmlt}. Finally, note that the B-derivative for the normal map \( (f_N)_S \), denoted by \(d(f_N)_S(\cdot)\), will take an analogous form to (\ref{q:bdivtrue}).

We shall use \(\Sigma_N\) to denote the sample covariance matrix of \(\left\{F(x_N,\xi_i)\right\}_{i=1}^N\), where \(x_N\) is an SAA solution to be formally defined in Theorem \ref{t:asy_dis}. Under Assumptions \ref{assu1} and \ref{assu2}, \(\Sigma_N\)  converges almost surely to \(\Sigma_0\), see \cite[Lemma 3.6]{lu:crc}. This combined with Assumption \ref{assu3} implies that for almost every \(\omega \in \Omega\) there exists an \(N_{\omega}\) such that \(\Sigma_N(\omega)\) is invertible for \(N \ge N_{\omega}\).

From Assumption \ref{assu4} it follows that \(f_N\) converges to \(f_0\) in probability at an exponential rate, as shown in \cite[Theorem 4]{l.b:acr} which is based on a general result \cite[Theorem 7.67]{shap.dent.rus:lsp}. That is, there exist positive real numbers \(\beta_1, \mu_1, M_1\) and \(\sigma_1\), such that the following holds for each \(\epsilon >0\) and \(N\):
\begin{equation}\label{q:funexpcon}
\Pr \left( \|f_N-f_0\|_{1,X} \ge \epsilon \right) \le \beta_1 \exp \left\{-N\mu_1\right\} +\frac{M_1}{\epsilon^n}\exp \left\{-\frac{N\epsilon^2}{\sigma_1} \right\}.
\end{equation}
Finally, note that Assumption 4 implies all conditions in Assumption 1; we put Assumption 1 as a separate assumption because some intermediate results do not require the stronger Assumption 4.

The following theorem is adapted from \cite[Theorem 7]{l.b:acr} and provides results relating to the asymptotic distribution of solutions to (\ref{q:saa}) and (\ref{q:nm_tform}).
\begin{theorem}\label{t:asy_dis}
Suppose that Assumptions \ref{assu1} and \ref{assu2} hold. Let $Y_0$ be a normal random vector in $\Rset^n$ with zero mean and covariance matrix $\Sigma_0$. Then there exist neighborhoods $X_0$ of $x_0$ and $Z$ of $z_0$ such that the following hold. For almost every $\omega \in \Omega$, there exists an integer $N_{\omega}$, such that for each $N \ge N_{\omega}$, the equation
(\ref{q:nm_sform})
 has a unique solution $z_N$ in $Z$, and the variational inequality (\ref{q:saa}) has a unique solution in $X_0$ given by $x_N=\Pi_S(z_N)$. Moreover, $\lim\limits_{N\to\infty}z_N = z_0$ and $\lim\limits_{N\to\infty}x_N = x_0$ almost surely,
\begin{equation}\label{q:zN_dist_cov}
\sqrt{N}(z_N - z_0) \Rightarrow (L_{K_0})^{-1}(Y_0),
\end{equation}
\begin{equation}\label{q:zN_dist_cov2}
\sqrt{N} L_{K _0}(z_N- z_0) \Rightarrow Y_0,
\end{equation}
and
\begin{equation}\label{q:xN_dist_cov}
\sqrt{N}(\Pi_S(z_N) -\Pi_S( z_0)) \Rightarrow \Pi_{K_0}\circ (L_{K_0})^{-1}(Y_0).
\end{equation}

Suppose in addition that Assumption \ref{assu4} holds. Then there exist positive real numbers \( \epsilon_0, \beta_0, \mu_0, M_0\) and \(\sigma_0\), such that the following holds for each \( \epsilon \in (0,\epsilon_0]\) and each \(N\):
\begin{eqnarray}
 && \Pr\left( \|x_N-x_0\| < \epsilon \right)  \ge \Pr\left( \|z_N-z_0\| < \epsilon \right) \nonumber \\[-1.5ex]
\label{q:exp_conv}\\[-1.5ex]
 && \ge 1-\beta_0\exp\left\{-N\mu_0\right\}-\frac{M_0}{\epsilon^n}\exp \left\{ \frac{-N\epsilon^2}{\sigma_0}\right\}.  \nonumber
\end{eqnarray}
\end{theorem}

The asymptotic distributions in (\ref{q:zN_dist_cov}), (\ref{q:zN_dist_cov2}) and (\ref{q:xN_dist_cov}) depend on \(z_0\) through \(\Sigma_0\), \(L_{K_0}=d(f_0)_S(z_0)\) and \(\Pi_{K_0}=d\Pi_S(z_0)\). How to estimate these functions using \(z_N\)  and the sample data requires special attention, since \(d\Pi_S(z_N)\) does not always converge to \(d\Pi_S(z_0)\). While \(d\Pi_S(\cdot)\) is the same function for all points in the relative interior of a cell in the normal manifold of \(S\) \cite[Section 5.2]{l.b:acr}, the function changes dramatically across different cells. In particular, if \(z_0\in \ri C_j\), where \(C_j\) is a \(k\)-cell in the normal manifold of \(S\) with \(k < n\), then \( d\Pi_S(z_0)\) is piecewise linear with multiple pieces. In contrast, as shown in \cite[Proposition 3.5]{lu:crc}, under Assumptions \ref{assu1} and \ref{assu2}, the probability of \(d\Pi_S(z_N)\) and \( d(f_N)_S(z_N) \) being linear maps goes to one as the sample size \(N\) goes to infinity. Thus, as long as $z_0$ does not belong to the interior of an $n$-cell in the normal manifold of $S$, \(d\Pi_S(z_N)\) does \emph{not} converge to \( d\Pi_S(z_0)\).

In \cite{l:acr} this issue was addressed by choosing a point near \(z_N\), but not necessarily $z_N$ itself, to use in the estimate for \(d\Pi_S(\cdot)\). To choose such a point, for each cell \(C_i\) in the normal manifold of \(S\) define a function \(d_i:\mathbb{R}^n \rightarrow \mathbb{R}\) by
\begin{equation}\label{q:nmcelldist}
d_i(z)=d(z,C_i)=\min\limits_{x \in C_i}\|x-z\|,
\end{equation}
and a function \( \Psi_i:\mathbb{R}^n\rightarrow \mathbb{R}^n\) by
\begin{equation}\label{q:cellbdiv}
\Psi_i(\cdot)=d\Pi_S(z)(\cdot) \mbox{  for any } z \in \ri C_i.
\end{equation}
In (\ref{q:nmcelldist}) any norm for vectors in \(\mathbb{R}^n\) can be chosen, and in (\ref{q:cellbdiv}) any \(z \in \ri C_i\) can be chosen since \(d\Pi_S(z)\) is the same function on the relative interior of a cell.
Next, choose a function \(g:\mathbb{N} \rightarrow \mathbb{R} \)  satisfying
\begin{enumerate}
\item \(g(N)> 0\) for each \(N \in \mathbb{N}\).
\item \(\lim\limits_{N \rightarrow \infty} g(N)= \infty\).
\item \(\lim\limits_{N \rightarrow \infty} \frac{N}{g(N)^2}= \infty\).
\item \(\lim\limits_{N \rightarrow \infty} g(N)^n \exp \left\{ -\sigma_0\frac{N}{(g(N))^2} \right\}=0\) for \(\sigma_0=\min \left\{ \frac{1}{4\sigma_0},\frac{1}{4\sigma_1},\frac{1}{4\sigma_0(E[C])^2}\right\}\), where \(\sigma_0,\) and \(\sigma_1\) are as in (\ref{q:funexpcon}) and (\ref{q:exp_conv}) respectively and \(C\) as in Assumption \ref{assu1}.
\item \(\lim\limits_{N \rightarrow \infty}\frac{N^{n/2}}{ g(N)^n} \exp \left\{ -\sigma g(N)^2 \right\}=0\) for each positive real number \( \sigma\).
\end{enumerate}
Note that \(g(N)=N^p\) for any \(p \in (0,1/2)\) satisfies 1--5.

Now for each integer \(N\) and any point \(z \in \mathbb{R}^n\), choose an index \(i_0\) by letting \(C_{i_0}\) be a cell that has the smallest dimension among all cells \(C_i\) such that \(d_i(z) \le 1/g(N)\). Then define functions \( \Lambda_N(z):\mathbb{R}^n \rightarrow \mathbb{R}^n\) by
\begin{equation}\label{q:bdivest}
\Lambda_N(z)(h)=\Psi_{i_0}(h),
\end{equation}
 and \(\Phi_N:\Pi^{-1}_S(O)\times \mathbb{R}^n \times \Omega \rightarrow \mathbb{R}^n\) by
\begin{equation}\label{q:finest}
\Phi_N(z,h,\omega)=df_N(\Pi_S(z))(\Lambda_N(z)(h))+h-\Lambda_N(z)(h).
\end{equation}
Moving forward we will be interested in \(\Phi_N(z_N(\omega),h,\omega)\), which for convenience we will express as \(\Phi_N(z_N)(h)\) with the \(\omega\) suppressed. We shall use \(z_N^*\) to denote a point in the relative interior of the cell \(C_{i_0}\) associated with $(N, z_N)$. With this notation it follows that \(d\Pi_S(z_N^*)=\Psi_{i_0}\) and
\begin{equation}\label{q:funcent}
\Phi_N(z_N)(h)=df_N(\Pi_S(z_N))(d\Pi_S(z_N^*)(h))+h-d\Pi_S(z_N^*)(h).
\end{equation}
As shown by Theorem \ref{t:asy_est} below, the function $\Lambda_N(z_N)$, which is the same as $d\Pi_S(z_N^*)$, provides a nice estimate for $d \Pi_S(z_0)$. The reason behind this result is the following. From (\ref{q:exp_conv}), there is a high probability for the collection of cells in the normal manifold of $S$ that are within a distance of $1/g(N)$ from $z_N$ to coincide with the collection of cells that contain $z_0$. Whenever this happens, $C_{i_0}$ is the cell that contains $z_0$ in its relative interior, and the two points $z_N^*$ and $z_0$ belong to the relative interior of the same cell $C_{i_0}$, with $d\Pi_S(z_N^*)=d\Pi_S(z_0)$. This observation will be used in the proofs of Theorems \ref{t:meth2} and \ref{t:icim} (with the definition of event $A_N$ in (\ref{q:consubcon})). Theorem \ref{t:asy_est} below was proved in \cite[Corollaries 3.2 and 3.3]{l:acr}.

\begin{theorem}\label{t:asy_est}
Suppose that Assumptions \ref{assu2} and \ref{assu4} hold. For each \( N \in \mathbb{N}\), let \( \Lambda_N\) and \(\Phi_N\) be as defined in (\ref{q:bdivest}) and (\ref{q:finest}). Then
\begin{equation}\label{q:bdcon}
\lim\limits_{N \rightarrow \infty} \Pr \left[ \Lambda_N(z_N)(h)=d\Pi_S(z_0)(h) \mbox{  for all } h \in \mathbb{R}^n\right] =1,
\end{equation}
and there exists a positive real number \(\theta\), such that
\begin{equation}\label{q:phiconv}
\lim\limits_{N \rightarrow \infty} \Pr \left[ \sup\limits_{h \in \mathbb{R}^n, h\ne 0} \frac{ \| \Phi_N(z_N)(h)-d(f_0)_S(z_0)(h)\|}{\| h\|} < \frac{\theta}{g(N)} \right] =1.
\end{equation}
Moreover suppose Assumption \ref{assu3} holds, and let \(\Sigma_N\) be as defined above. Then
\[
\sqrt{N} \Sigma_0^{-1/2}\Phi_N(z_N)(z_N-z_0) \Rightarrow \mathcal{N}(0,I_n),
\]
and
\begin{equation}\label{q:asy_est}
\sqrt{N} \Sigma_N^{-1/2}\Phi_N(z_N)(z_N-z_0) \Rightarrow \mathcal{N}(0,I_n).
\end{equation}
\end{theorem}
In contrast to (\ref{q:zN_dist_cov}) and (\ref{q:zN_dist_cov2}), the quantities in (\ref{q:asy_est}) are computable using only the sample data, providing a basis for building confidence regions of $z_0$. Additionally, (\ref{q:dSnK}) and (\ref{q:bdcon}) suggest the use of \(\Lambda_N\) as an estimate for \(\Pi_{K_0}\) when developing methods for building confidence intervals for $x_0$. Similar results were shown in \cite{l.b:acr} but with \(\Lambda_N\) taken to be a weighted average of all the functions \( \Psi_i\) satisfying \(d_i(z_N) \le 1/g(N) \).

In \cite{lu:crc} a different tack was taken on constructing confidence regions. Instead of estimating functions that converge to \(d(f_0)_S(z_0)\), it was shown that under Assumptions \ref{assu1} and \ref{assu2}, the difference of \(-\sqrt{N}d(f_N)_S(z_N)(z_0-z_N)\) and \(\sqrt{N}d(f_0)_S(z_0)(z_N-z_0)\) converges to zero in probability, and consequently that
\[
 -\sqrt{N}d(f_N)_S(z_N)(z_0-z_N) \Rightarrow Y_0.
\]
   Because \(d(f_N)_S(z_N)\) is a linear function with high probability, even when \(d(f_0)_S(z_0)\) is piecewise linear, the above expression provides an easier method to calculate confidence regions and simultaneous confidence intervals.

As noted earlier, confidence regions do not directly lead to useful individual confidence intervals. The papers \cite{l:acr} and \cite{l.b:acr} did not discuss how to compute individual confidence intervals, while \cite{lu:crc} provided a method for such computation. Below we briefly introduce the latter method.

 With the notation used above (\ref{q:bdivtrue}), let $\Gamma$ denote the normal manifold of $S$ and $\Gamma'(z_0)$ denote the conical subdivision that corresponds to  \( d\Pi_S(z_0)\), which is also the conical subdivision for \( d(f_0)_S(z_0)\). Suppose \(\Gamma'(z_0) =\left\{K_1,\dots,K_k\right\}\). Then for each \(i=1,\dots,k\), the restriction of \(d(f_0)_S(z_0)\) on $K_i$, which we denote by  \(\allowbreak d(f_0)_S(z_0)|_{K_i}\), coincides with a linear function; let $M_i$ be the matrix representing that linear function. Moreover under Assumption \ref{assu2}, \(d(f_0)_S(z_0)\) is a global homeomorphism so each matrix \(M_i\) is invertible. We then define \(Y^i=M_i^{-1}Y_0\). Since \(Y_0\) is a multivariate normal random vector each \(Y^i\) is a multivariate normal random vector with covariance matrix \(M_i^{-1}\Sigma_0M_i^{-T}\).

We define the number \( r_j^i=\sqrt{(M_i^{-1}\Sigma_0M_i^{-T})_{jj}} \) for each \(i=1,\dots,k\) and \(j=1,\dots,n\). Finally for each \(\alpha \in (0,1)\) let \(\chi^2_1(\alpha)\) be the \((1-\alpha)^{\tiny\mbox{th}}\) percentile of a \(\chi^2\) random variable with one degree of freedom. It then follows that
\[
\Pr\left( |(Y^i)_j| \le r_j^i \sqrt{\chi^2_1(\alpha)} \right)=1-\alpha.
\]
The following theorem on individual confidence intervals for components of \(z_0\) was proven in \cite[Theorem 5.1]{lu:crc}.

\begin{theorem}\label{t:indold}
Suppose that Assumptions \ref{assu1}, \ref{assu2} and \ref{assu3} hold. Let \(K_i, M_i, Y^i\) and \(r_j^i\) be defined as above. For each integer \(N\) with \(d(f_N)_S(z_N)\) being an invertible linear map, define a number
\[
r_{Nj}=\sqrt{(d(f_N)_S(z_N)^{-1}\Sigma_Nd(f_N)_S(z_N)^{-T})_{jj} }
\]
for each \(j=1,\dots,n\). Let $r_{Nj}=0$ if  \(d(f_N)_S(z_N)\) is not an invertible linear map. Then for each real number \(\alpha \in (0,1)\) and for each \(j=1,\dots,n\),
\begin{eqnarray}
 && \lim\limits_{N \rightarrow \infty}\Pr\left( \frac{\sqrt{N}|(z_n-z_0)_j|}{r_{Nj}}\le \sqrt{\chi^2_1(\alpha)}  \right)  \nonumber \\
=&&\sum\limits_{i=1}^{k} \Pr\left( \Big|\frac{(Y^i)_j}{r^i_j}\Big| \le \sqrt{\chi^2_1(\alpha)} \mbox{  and  } Y^i \in K_i \right) \label{q:inold}
\end{eqnarray}
Moreover, suppose for a given \(j=1,\dots,n\) that the equality
\[
\Pr\left( \Big|\frac{(Y^i)_j}{r^i_j}\Big| \le \sqrt{\chi^2_1(\alpha)} \mbox{  and  } Y^i \in K_i \right)=\Pr\left( \Big|\frac{(Y^i)_j}{r^i_j}\Big| \le \sqrt{\chi^2_1(\alpha)}\right)\Pr\left( Y^i \in K_i \right)
\]
holds for each \(i=1,\dots,k\). Then for each real number \( \alpha \in (0,1)\),
\[
\lim\limits_{N \rightarrow \infty} \Pr \left( |(z_N-z_0)_j| \le \frac{\sqrt{\chi^2_1(\alpha)} r_{Nj}}{\sqrt{N}} \right)=1-\alpha.
\]
\end{theorem}

We see in (\ref{q:inold}) that this method of constructing individual confidence intervals, while easily computable using only the sample data, produces intervals whose asymptotic level of confidence is dependant on the true solution, unless the condition below (\ref{q:inold}) is satisfied. The latter condition is satisfies, when \(d(f_0)_S(z_0)\) is a linear function or has only two selection functions, in which case the intervals computed from this method will be asymptotically exact. In general, however, the level of confidence for such intervals cannot be guaranteed. This limitation motivates the development of methods proposed in the following section.

\section{New methods for building individual confidence intervals}\label{s:theory}

In this section we present three new methods for building individual confidence intervals. The first two methods produce intervals for $(z_0)_j$, that have a specified level of confidence for situations more general than the method examined in Theorem \ref{t:indold}. Those two methods rely on the estimate \(\Phi_N(z_N)\); when \(\Phi_N(z_N)\) is a linear function, they return the same interval as the method examined in Theorem \ref{t:indold}. The methods differ when \(\Phi_N(z_N)\) is piecewise linear. The first method (given in Theorem \ref{t:mainres}) uses all selection functions of \(\Phi_N(z_N)\) to calculate an interval. The second (given in Theorem \ref{t:meth2}) uses \(z_N\) to determine a subset of selection functions to be used in an interval's computation. When the set \(S\) is a box these intervals can be projected onto \(S\) to produce intervals that cover \((x_0)_j\) at a rate at least as large as the coverage rate of \((z_0)_j\) by the initial intervals.

The third method (given in Theorem \ref{t:icim}) considers the computation of individual confidence intervals for $x_0$ directly. This method estimates the function that appears in the right-hand of (\ref{q:xN_dist_cov}) by using both the  function \(\Lambda_N\) as defined in (\ref{q:bdivest}) and the function \(\Phi_N(z_N)\). Initially these two functions are considered separately, and the relation between $x_N$ and $z_N$ is used to emulate the approach of the second method. When calculating an interval's length, with high probability one only need to consider a single selection function of the estimate constructed from \(\Lambda_N\) and \(\Phi_N(z_N)\).

\subsection{The first method (an indirect approach)}\label{s:nmicim1}

In this method, we compute confidence intervals for $(z_0)_j$, for each $j=1,\cdots,n$, based on equation (\ref{q:result}) in Theorem \ref{t:mainres}. In that equation, $r$ is an arbitrarily chosen real number, and $a^r(\Phi_N^{-1}(z_N)\Sigma_N^{1/ 2} )_j)$ returns a number determined by the $j$th component of the function $\Phi_N^{-1}(z_N)\Sigma_N^{1/ 2}$. In the following, we start with the definition of $a^r(\cdot)$.

Let \( \psi : \mathbb{R}^n \to \mathbb{R} \) be a continuous function, and \(Z \sim \mathcal{N}(0,I_{n})\). Suppose that \(\allowbreak \Pr\left(\psi(Z) = b  \right) =0 \) for all \(b\) and \(\Pr\left( \beta_1 < \psi(Z)< \beta_2 \right) >0 \) for all \( \beta_1 < \beta_2\). Then given any \( \alpha \in (0,1) \) and \(  r \in \mathbb{R}\) there exists a unique point \(a^r(\psi) \in (0,\infty) \) such that
\[
\Pr\left(-a^r(\psi) \le \psi(Z)-r \le a^r(\psi)   \right) =1- \alpha.
\]
Let  \( \alpha \in (0,1) \) be fixed. For any function \(f:\mathbb{R}^n \to \mathbb{R} \), define
\begin{equation}\label{q:alphdef}
a^r(f)=\inf\{ l \ge 0 | \Pr\left(-l \le f(Z)-r \le  l \right) \ge 1- \alpha \}.
\end{equation}
It then follows that
 \begin{enumerate}
 \item \(a^r(f) < \infty \).
 \item\(\Pr\left(-a^r(f) \le f(Z)-r \le a^r(f)\right) \ge 1- \alpha\).
 \item \(\Pr\left(-(a^r(f)-\delta) \le f(Z)-r \le a^r(f)-\delta\right ) < 1- \alpha \) for all \( \delta >0\).
\end{enumerate}

In the proof of Theorem \ref{t:mainres} we use the following two lemmas.
\begin{lemma}\label{t:lem1}
Let \(\psi \) be as above and \(\left\{\psi_N\right\}_{N=1}^{\infty}\) be a sequence of functions that converges pointwise to  \(\psi \). Then for any \( r \in \mathbb{R} \), \(\lim_{N\rightarrow \infty} a^r(\psi_N) = a^r(\psi)  \).
\end{lemma}
\begin{proof}
Note \( \sup_{N} a^r(\psi_N) < \infty\). This follows from the fact that \(\psi_N(Z)\) converges to \( \psi(Z) \) a.s. and so \( \left\{ \psi_N(Z)\right\}_{N=0}^{\infty} \) is tight. Next fix a subsequence, again indexed by \(N\), along which \( a^r(\psi_N) \to a^{\ast} \). It suffices to show \( a^{\ast} = a^r(\psi) \).

Note that \( a^{\ast} \ne 0\). If this were the case then for every \( \epsilon > 0 \)
\[
1-\alpha \le \lim\limits_{N\rightarrow \infty} \Pr \left( -\epsilon \le \psi_N(Z)-r \le \epsilon \right) = \Pr\left( -\epsilon \le \psi(Z)-r \le \epsilon \right) .
\]
Since \( \epsilon \) is arbitrary this would imply \( \Pr \left( \psi(Z) = r \right) \ge 1-\alpha \), a contradiction.

Assume now without loss of generality that \( \inf_{N} a^r(\psi_N) > 0 \). Then
\begin{equation}\label{q:lem1p1}
1-\alpha \le \lim\limits_{N\rightarrow \infty} \Pr\left( -1 \le  \frac{\psi_N(Z)-r}{a^r(\psi_N)}  \le 1 \right) =\Pr\left( -1 \le  \frac{\psi(Z)-r}{a^{\ast}}  \le 1 \right).
\end{equation}
Applying the same argument for all \( 0<\delta < \inf_N a^r(\psi_N) \) we see that
\[
\Pr \left( -1 \le \frac{\psi(Z)-r}{(a^{\ast}-\delta)} \le 1 \right) \le 1-\alpha .
\]
Sending \( \delta\) to \( 0 \) we obtain \( \Pr \left( -a^{\ast} \le \psi(Z)-r \le a^{\ast} \right) \le 1-\alpha \), which combined with (\ref{q:lem1p1}) gives
\[
\Pr \left( -a^{\ast}\le \psi(Z)-r \le a^{\ast} \right)=1-\alpha.
\]
Thus \( a^{\ast}= a^r(\psi)\), and \(\lim_{N\rightarrow \infty} a^r(\psi_N) = a^r(\psi)  \).
\end{proof} \qed

Let \(C(\mathbb{R}^n,\mathbb{R})\) denote the space of continuous functions from \(\mathbb{R}^n\) to \(\mathbb{R}\). Equipped with the local uniform topology, this is a Polish space.
\begin{lemma}\label{t:lem2}
Let \( \left\{ \psi_N \right\}_{N=1}^{\infty}\) be a sequence of \(C(\mathbb{R}^n,\mathbb{R})\) valued random variables which converges in distribution to  \(\psi \). Also let \(\left\{Z_N\right\}_{N=1}^{\infty}\) be a sequence of \(\mathbb{R}^n\) valued random variables converging in distribution to \( Z \). Then for any $r\in \Rset$,
 \[
\Pr \left( -a^r(\psi_N) \le \psi_N(Z_N)-r \le a^r(\psi_N) \right) \to 1-\alpha.
\]
\end{lemma}
\begin{proof}
By Lemma \ref{t:lem1} and the convergence of \(\psi_N\) to \(\psi\), it follows that \( a^r(\psi_N) \to a^r(\psi) \) in probability. Also since \( a^r(\psi) > 0\),
 \[
\frac{1}{a^r(\psi_N)} \mathds{1}_{a^r(\psi_N)>0} \to \frac{1}{a^r(\psi)}
\]
in probability, where \(\mathds{1}_{a^r(\psi_N)>0}\) is the indicator random variable for the event \({a^r(\psi_N)>0}\).  Let \( A_N\) denote the event that \(a^r(\psi_N) >0 \). Then
\begin{eqnarray}
\Pr\left( -a^r(\psi_N) \le \psi_N(Z_N)-r \le a^r(\psi_N)  \right) &&=\Pr \left( A_N; \; -1 \le \frac{\psi_N(Z_N)-r}{a^r(\psi_N)} \le 1  \right)\nonumber\\
&&+\Pr \left( A_N^c; \; -a^r(\psi_N) \le \psi_N(Z_N)-r \le a^r(\psi_N) \right). \nonumber
\end{eqnarray}
By \( a^r(\psi_N) \to a^r(\psi) \) in probability and \( a^r(\psi) >0\), it follows that \(\Pr \left(A_N \right) \to 1\). Therefore,
\[
\Pr \left(A_N^c;\; -a^r(\psi_N) \le \psi_N(Z_N)-r \le a^r(\psi_N) \right) \to 0 \mbox{ as } N\to \infty.
\]
Let \(B_N\) be the event that \(-1\le \frac{\psi_n(Z_N)-r}{a^r(\psi_N)}\mathds{1}_{a^r(\psi_N)>0} \le 1\). By the convergence of \(\psi_N\) to \(\psi \) and \(Z_N\) to \( Z \), we have that \(\psi_N(Z_N) \Rightarrow \psi(Z) \) in distribution, and thus
\[
\Pr \left( B_N \right) \to \Pr\left( -1 \le \frac{\psi(Z)-r}{a^r(\psi)} \le 1\right) = \Pr \left( -a^r(\psi) \le \psi(Z)-r \le a^r(\psi) \right)=1-\alpha.
\]
Consequently, \( \Pr\left( -a^r(\psi_N) \le \psi_N(Z_N)-r \le a^r(\psi_N)  \right)  \to 1-\alpha. \)
\end{proof} \qed

The application of these lemmas to our problem of interest is facilitated by the following two propositions.

\begin{proposition}\label{t:convprop}
(a) Let \( f:\Rset^n\to \Rset^n \) be a piecewise linear function and \( \{f_N\}_{N=1}^{\infty} \) a sequence of piecewise linear functions from $\Rset^n$ to $\Rset^n$ with
\begin{equation}\label{q:convprp1}
\sup_{h\in \mathbb{R}^n, h \ne 0} \frac{\|f_N(h)-f(h) \|}{\| h\| } \to 0.
\end{equation}
Suppose that there exists a conical subdivision \( \Gamma=\{\gamma_1,\gamma_2 \dots \gamma_m \} \) of \( \mathbb{R}^n \) such that for all \( N \) sufficiently large \( f_N|_{\gamma_i}=A_{N,i} \) and \( f|_{\gamma_i}=A_i \) are linear functions for each \( \gamma_i \). Then
\begin{equation}\label{q:convprp2}
\sup_{h\in \mathbb{R}^n, h \ne 0} \frac{\|A_{N,i}h-A_ih \|}{\| h\| } \to 0 \mbox{ for } i=1,\dots,m.
\end{equation}

(b) Suppose in addition that \(f\) is a homeomorphism. Then for all \(N \) sufficiently large \(f_N\) is a homeomorphism and \( f_N^{-1} \) converges uniformly on compacts to \(f^{-1} \).
\end{proposition}
\begin{proof}
By (\ref{q:convprp1}), $\sup_{h\in \gamma_i, h \ne 0} \frac{\|A_{N,i}h-A_ih \|}{ \| h\| } $ converges to 0 as $N\to \infty$, for each $i=1,\dots,m$. As $\Gamma$ is a conical subdivision of $\mathbb{R}^n$, $\gamma_i$ is of dimension $n$ which means that it contains a ball in $\Rset^n$. The fact that $\|A_{N,i}h-A_ih \|$ converges to 0 for all $h$ in a ball implies that the matrix $A_{N,i}$ converges to $A_i$, giving (\ref{q:convprp2}).


To prove (b) first note that since \(f\) is a homeomorphism, \( A_i^{-1} \) is well defined for each \(i\) and \( \left\{ A_1^{-1}, A_2^{-1}, \dots, A_m^{-1} \right\} \) provides a family of selection functions for \(f^{-1}\) \cite[Proposition 2.3.2]{sch:ipdf}. Moreover we have that \(f^{-1} \) is Lipschitz continuous with Lipschitz constant
\[
\delta = \max_{1 \le i \le m} \left( \|A_i^{-1}\| \right) < \infty.
\]

Similarly for \(N\) sufficiently large the functions \(f_N-f\) will be piecewise linear with a family of selection functions given by \( \left\{ A_{N,1}-A_1, \dots , A_{N,m}-A_m \right\} \), and thus Lipschitz continuous with Lipschitz constant
\[
\rho_N = \max_{1 \le i \le m} \left( \|A_{N,i}-A_i\| \right)
\]
From part \((a)\) we have \(\lim_{N \to \infty}\|A_{N,i}-A_i\| =0\) for each \(i \), so for all \(N\) sufficiently large \( \rho_N < \delta^{-1}\). From \cite[Lemma 3.1]{rob:ift} it then follows that \(f_N\) is a homeomorphism for \(N\) sufficiently large.

To obtain \(f_N^{-1} \to f^{-1} \) uniformly on compacts, note first from \(\lim_{N \to \infty}A_{N,i}^{-1} =A_i^{-1}\) it follows that \( \{f_N^{-1} \}_{N=v}^{\infty} \) is uniformly Lipschitz continuous for \(v\) large enough. Then for any compact set \( C \) and any subsequence of \( f_N^{-1} \) there exists a further subsequence, \( f_{N_k}^{-1} \) that converges uniformly on \(C\) to some function \( g \). To prove part \( (b) \) it then suffices to show that \(g(x)=f^{-1}(x)\).

To see that this holds let \( x \in C , \alpha_k=f_{N_k}^{-1}(x), \) and \( \alpha=g(x)\). By \( \alpha_k \rightarrow \alpha \) and \( f_{N_k} \rightarrow f \) it follows that \( f_{N_k}(\alpha_k) \rightarrow f(\alpha) \). Also for each \(k\)
\[
f_{N_k}(\alpha_k)=f_{N_k}(f_{N_k}^{-1}(x))=x.
\]
Thus \(x=f(\alpha)=f(g(x)) \), or \(g(x)=f^{-1}(x)\), the desired result.
\end{proof} \qed

\begin{proposition}\label{t:sviprop}
Suppose that Assumptions \ref{assu2}, \ref{assu3} and \ref{assu4} hold, and for each \( N \in \mathbb{N} \) let \( \Phi_N(z_N) \) be as in (\ref{q:funcent}). Then \( \Phi_N^{-1}(z_N)\Sigma_N^{1/ 2} \) converges to \(d(f_0)_S^{-1}(z_0)\Sigma_0^{1/2} \) in probability, uniformly on compacts.
\end{proposition}
\begin{proof}
As previously noted, when Assumption \ref{assu4} holds the conditions of Assumption \ref{assu1} are satisfied, and under Assumptions \ref{assu1} and \ref{assu2} \(\Sigma_N\) converges almost surely to \(\Sigma_0\).
Convergence of \(\Sigma_N\) to \(\Sigma_0\) and (\ref{q:phiconv}) imply that for all \( \epsilon > 0\)
\begin{equation}\label{q:funconp}
 \lim_{N \to \infty}\Pr\left( \sup_{h \in \mathbb{R}^n, h \ne 0} \frac{\| \Sigma_N^{-1/ 2}\Phi_N(z_N)(h)-\Sigma_0^{-1/2}d(f_0)_S(z_0)(h) \|   }{\| h \|} < \epsilon   \right)=1.
\end{equation}
By a standard subsequential argument we can assume without loss of generality that almost surely
\[
\sup_{h \in \mathbb{R}^n, h \ne 0} \frac{\| \Sigma_N^{-1/ 2}\Phi_N(z_N)(h)-\Sigma_0^{-1/2}d(f_0)_S(z_0)(h) \|   }{\| h \|} \rightarrow 0.
\]
In order to show almost sure convergence of \(\Phi_N^{-1}(z_N)\Sigma_N^{1/ 2}\) to \( d(f_0)_S^{-1}(z_0)\Sigma_0^{1/2} \)  we will apply Proposition \ref{t:convprop}. It suffices then that for a.e. \( \omega\), with \(f_N=\Phi_N^{-1}(z_N(\omega))\Sigma_N^{1/ 2}(\omega)\) and \(f=d(f_0)_S^{-1}(z_0)\Sigma_0^{1/2}\), conditions of Proposition \ref{t:convprop} are satisfied.

To this end recall the expressions for \(d(f_0)_S(z_0)\) given in (\ref{q:bdivtrue}), \(\Phi_N(z_N)\) given in (\ref{q:finest}) and \(\Lambda_N(z_N)\) given in (\ref{q:bdivest}). From these it is clear that the conditions in part \((a)\) of Proposition \ref{t:convprop} will be satisfied if we can find a conical subdivision \(\Gamma\) such that for every \(\gamma_i\in \Gamma\) and \(z \in \mathbb{R}^n\), \(d\Pi_S(z)|_{\gamma_i}\) is equal to a linear function.

Let \(C_1,\dots,C_l\) be all of the \(k\)-cells in the normal manifold of \(S\), \(k=0,1,\dots,n\). Then for every \(z\in \mathbb{R}^n\), \(z\in \ri C_j\) for some \(j\), and \(d\Pi_S(z)(\cdot) = \Psi_j(\cdot)\) for \(\Psi_j\) defined as in (\ref{q:cellbdiv}). The desired subdivision \(\Gamma\) can be constructed by taking the collection of all cones with non-empty interior of the form \(\gamma=\cap_{k=1}^l \gamma_{k}\) where each \(\gamma_k\) is from a conical subdivision of \(\Psi_k\).

Finally by Assumptions \ref{assu2} and \ref{assu3}, \(\Sigma_0^{-1/2}d(f_0)_S(z_0)\) is a homeomorphism, satisfying the condition in part \((b)\) of Proposition \ref{t:convprop}. The result follows.
\end{proof} \qed

At this point we are able to present the main result for our first method on computation of asymptotically exact individual confidence intervals.
\begin{theorem}\label{t:mainres}
Suppose that Assumptions \ref{assu2}, \ref{assu3} and \ref{assu4} hold. Let \(\alpha \in (0,1)\), $r\in \Rset$, and let \( a^r( \cdotp ) \) be as defined in (\ref{q:alphdef}). Then for every \(j=1,\dots,n\),
\begin{equation}\label{q:result}
\lim\limits_{N\rightarrow \infty}\Pr\left( \big|\sqrt{ N}(z_N-z_0)_j-r\big|   \le a^r\left( ( \Phi_N^{-1}(z_N)\Sigma_N^{1/ 2} )_j \right) \right) = 1-\alpha.
\end{equation}
\end{theorem}
\begin{proof}
By Proposition \ref{t:sviprop}, \( (\Phi_N^{-1}(z_N)\Sigma_N^{1/ 2})_j \) converges to \((L_K^{-1}\Sigma_0^{1/2})_j \) in \(C(\mathbb{R}^n,\mathbb{R})\), in probability. Since \(L_K^{-1}\Sigma_0^{1/2}\) is a piecewise linear homeomorphism it follows that for  \(Z \sim N(0,I_{n})\) and each \(j=1,\dots,n\),
\[
\Pr\left( (L_K^{-1}\Sigma_0^{1/2})_j(Z) = b  \right) =0 \mbox{ for all } b
\]
and
\[
\Pr\left( \beta_1 < (L_K^{-1}\Sigma_0^{1/2})_j(Z)< \beta_2 \right) >0 \mbox{ for all } \beta_1 < \beta_2.
\]
Taking \( Z_N = \sqrt{N} \Sigma_N^{-1/2} \Phi_N (z_N)(z_N-z_0) \), by Theorem \ref{t:asy_est} (see (\ref{q:asy_est})) \( Z_N \) converges in distribution to \(Z \). Then with \(\psi_N= (\Phi_N^{-1}(z_N)\Sigma_N^{1/ 2})_j \), and \( \psi=(L_K^{-1}\Sigma_0^{1/2})_j \),  it follows from Lemma \ref{t:lem2} that
\begin{eqnarray}
&& \Pr \left( -a^r(\psi_N) \le \psi_N(Z_N)-r \le a^r(\psi_N) \right) \nonumber \\
=&& \Pr \left( -a^r(\psi_N) \le (\Phi_N^{-1}(z_N)\Sigma_N^{1/ 2})_j\big(\sqrt N \Sigma_N^{-1/2} \Phi_N (z_N-z_0)\big)-r \le a^r(\psi_N) \right) \nonumber \\
=&& \Pr \left( -a^r(\psi_N) \le \sqrt{N}(\Phi_N^{-1}(z_N)\Sigma_N^{1/ 2})_j\big( \Sigma_N^{-1/2} \Phi_N (z_N-z_0)\big)-r \le a^r(\psi_N) \right) \nonumber \\
=&& \Pr\left( -a^r\big((\Phi_N^{-1}(z_N)\Sigma_N^{1/ 2} )_j\big) \le \sqrt{N}(z_N-z_0)_j-r \le a^r\big((\Phi_N^{-1}(z_N)\Sigma_N^{1/ 2} )_j\big) \right) \nonumber 
\end{eqnarray}
converges to $1-\alpha$ as \(N \rightarrow \infty\).
\end{proof} \qed

While Theorem \ref{t:mainres} proves the asymptotic exactness of intervals for a general choice of \(r\), (\ref{q:result}) and (\ref{q:alphdef}) indicate how the choice of \(r\) will affect both an interval's center and length. Additionally, when $\Phi_N(z_N)$ is piecewise linear evaluating \( a^r\big((\Phi_N^{-1}(z_N)\Sigma_N^{1/ 2} )_j\big) \)  requires working with each selection function, which can pose a computational challenge if the number of selection functions is large. The second method limits the computational burden of working with a piecewise linear function by considering only a subset of selection functions indicated by \(z_N\).

\subsection{The second method (an indirect approach)}\label{s:nmicim2}

In this method, we compute confidence intervals for $(z_0)_j$, for each $j=1,\cdots,n$, based on equation (\ref{q:mainres}) in Theorem \ref{t:meth2}, in which $\eta^{\alpha}_j(\cdot,\cdot)$ replaces \(a^r(\cdot)\) in the first method to determine an interval's width. Below we give the definition of $\eta^{\alpha}_j(\cdot,\cdot;)$. Let \(f:\Rset^n \rightarrow \Rset^n\) be a piecewise linear homeomorphism with a family of selection functions \( \left\{ M_1,\dots,M_l \right\}\), and the corresponding conical subdivision \(\left\{ K_1,\dots,K_l\right\}\). As before, let \( (f)_j \) denote the \(j^{\tiny\mbox{th}}\) component function of \(f\). For any choice of cone \(K_i\), \(i=1,\dots,l\),  component \(j =1,\dots,n\) and \(\alpha \in (0,1)\) we first define \(\eta^{\alpha}_j(f,x)\) for points \(x \in \mbox{int} K_i\) as the unique and strictly positive number satisfying
\begin{equation}\label{q:afundef}
\Pr \left( | \left( f^{-1}(Z)\right)_j | \le \eta^{\alpha}_j(f,x), \;  f^{-1}(Z) \in K_i \right)=(1-\alpha)\Pr \left( f^{-1}(Z) \in K_i \right).
\end{equation}
Note that \(\eta^{\alpha}_j(f,x)\) is the same number for all \(x \in \mbox{int}K_i\), since nothing in the above definition depends on the exact location of $x$, except that $K_i$ has to be the cone containing $x$ in its interior. Because \(f\) is a homeomorphism we can rewrite (\ref{q:afundef}) as
\begin{equation}\label{q:adef2}
\Pr \left( | \left( M^{-1}_iZ\right)_j | \le \eta^{\alpha}_j(f,x), \; M^{-1}_iZ \in K_i \right)=(1-\alpha)\Pr \left( M^{-1}_iZ \in K_i \right).
\end{equation}
For points \(x \in \bigcap_{s=1}^k K_{i_s}\) define \(\eta^{\alpha}_j(f,x)=\max\limits_{s=1,\dots,k} \eta^{\alpha}_j(f,x_{i_s})\) where \(x_{i_s} \in \mbox{int}K_{i_s}\).

The following Lemma will play a similar role in the proof of Theorem \ref{t:meth2} as Lemma \ref{t:lem1} did in the proof of Theorem \ref{t:mainres}.
\begin{lemma}\label{t:m2dc}
Let \( \left\{ f_m \right\}_{m=1}^{\infty}\) be a sequence of piecewise linear functions such that for all \(m\) sufficiently large \(f_m\) and \(f\) have a common conical subdivision \( \left\{ K_1,\dots,K_l \right\} \), and
\[
 \sup\limits_{h \in \mathbb{R}^n,h \ne 0} \frac{\| f_m(h)-f(h)\| }{\|h\|} \rightarrow 0.
\]
Then for all \(m\) sufficiently large \(f_m\) will be a homeomorphism and for all \(\alpha \in (0,1)\), \(x\in \Rset^n\) and \(j=1,\dots,n\)  one has \( \eta^{\alpha}_j(f_m,x) \rightarrow \eta^{\alpha}_j(f,x)\).
\end{lemma}
\begin{proof}
From Proposition \ref{t:convprop} it follows that \(f_m\) will be a homeomorphism for all \(m\) sufficiently large. The convergence of   \( \eta^{\alpha}_j(f_m,x)\) to \( \eta^{\alpha}_j(f,x)\) can be shown using an argument analogous to the one used in the proof of Lemma \ref{t:lem1} and is therefore omitted.
\end{proof} \qed

In the proof of Theorem \ref{t:meth2} we make use of the notation introduced before Theorem \ref{t:indold}. With this notation \(\Gamma'(z_0)=\left\{K_1,\dots,K_k\right\}\)  is the conical subdivision associated with \(d(f_0)_S(z_0)\) such that \(d(f_0)_S(z_0)|_{K_i}=M_i\)  and \(K_i=\cone(P_i-z_0)\) where \(P_1,\dots,P_k\) are all \(n\)-cells in the normal manifold of \(S\) that contain \(z_0\). Note that for \(i=1,\dots,k\) we can write \(Y^i=M_i^{-1}\Sigma_0^{1/2}Z\) and \(Y_0=\Sigma_0^{1/2}Z\) where \( Z \sim \mathcal{N}(0,I)\). Finally we define \(Y^*=d(f_0)_S^{-1}(z_0)\Sigma_0^{1/2}Z\), and note that \(Y^*\mathds{1}_{Y^*\in K_i}=Y^i\mathds{1}_{Y^i\in K_i}\).

\begin{theorem}\label{t:meth2}
Let Assumptions 2, 3 and 4 hold. Then with \(\Phi_N(z_N)(\cdot)\) and \(z_N^*\) as defined in (\ref{q:funcent}) one has that for all \(j=1,\dots,n\) and \(\alpha \in (0,1)\),
\begin{equation}\label{q:mainres}
\Pr\left(\sqrt{N}|(z_N-z_0)_j| \le \eta^{\alpha}_j(\Sigma_N^{-1/2}\Phi_N(z_N),z_N-z_N^*) \right) \rightarrow 1-\alpha.
\end{equation}
\end{theorem}
\begin{proof}
Let \(C_i\), \(i=1,\dots,l\) be all of the cells in the normal manifold of \(S\), and for each \(N\) define the event
 \begin{equation}\label{q:consubcon}
A_N= \bigg\{ \omega \bigg| \big\{i | d_i(z_N(\omega)) \le 1/g(N)\big\} = \big\{ i | z_0 \in C_i \big\}\bigg\}.
\end{equation}
  By the remarks below (\ref{q:funcent}), if \( \omega \in A_N\) then the two points $z_N^*$ and $z_0$ belong to the relative interior of the same cell in the normal manifold of $S$, with \(\Gamma'(z_0)=\Gamma'(z_N^*(\omega))\) and  \(d(f_0)_S(z_0)\) and \(\Phi_N(z_N(\omega))\) sharing the conical subdivision \( \left\{K_1,\dots,K_k \right\}\). Moreover as shown in \cite[Theorem 3.1]{l:acr} \(\lim_{N \rightarrow \infty} \Pr \left( A_N \right) =1\), so it follows from (\ref{q:funconp})
\begin{equation}\label{q:tconeq}
 \lim_{N \to \infty}\Pr\left(A_N; \; \sup_{h \in \mathbb{R}^n, h \ne 0} \frac{\| \Sigma_N^{-1/ 2}\Phi_N(z_N)(h)-\Sigma_0^{-1/2}d(f_0)_S(z_0)(h) \|   }{\| h \|} < \epsilon   \right)=1.
\end{equation}
Combining this with Lemma \ref{t:m2dc} it follows that for all fixed \(x\), \(\eta^{\alpha}_j(\Sigma_N^{-1/2}\Phi_N(z_N),x)\) converges in probability to \(\eta^{\alpha}_j(\Sigma_0^{-1/2}d(f_0)_S(z_0),x)\).

Next let \(B\) be a fixed neighborhood of \(z_0\) such that \(B\cap( z_0+ K_i) = B \cap P_i\) for \(i=1,\dots,k\). We then have
\begin{eqnarray}
&& \lim\limits_{N\rightarrow \infty} \Pr \left(\sqrt{N}|(z_N-z_0)_j| \le \eta^{\alpha}_j(\Sigma_N^{-1/2}\Phi_N(z_N),z_N-z_N^*) \right) \nonumber \\
=&&  \lim\limits_{N\rightarrow \infty} \Pr \left(\sqrt{N}|(z_N-z_0)_j| \le \eta^{\alpha}_j(\Sigma_N^{-1/2}\Phi_N(z_N),z_N-z_N^*); \;  A_N\right) \nonumber \\
=&&  \lim\limits_{N\rightarrow \infty} \sum\limits_{i=1}^k\Pr \left(\sqrt{N}|(z_N-z_0)_j| \le \eta^{\alpha}_j(\Sigma_N^{-1/2}\Phi_N(z_N),z_N-z_N^*); \;  A_N; \;  z_N \in B \cap \mbox{int}P_i\right) \nonumber \\
=&&  \lim\limits_{N\rightarrow \infty} \sum\limits_{i=1}^k\Pr \left(\sqrt{N}|(z_N-z_0)_j| \le \eta^{\alpha}_j(\Sigma_N^{-1/2}\Phi_N(z_N),x_i); \;  A_N; \; z_N \in B \cap \mbox{int}P_i\right) \nonumber
\end{eqnarray}
where \(x_i\) is any point in \(\mbox{int}K_i\). The first equality above follows from \( \lim_{N\rightarrow \infty} \Pr\left(A_N\right)=1\), and the second from \(\lim_{N\rightarrow \infty} \Pr\left( z_N \in \Rset^n \backslash \cup_{i=1}^k B \cap\mbox{int}P_i\right)=0\) as shown in \cite[Proposition 3.5]{lu:crc}. For the final equality, recall that $\omega \in A_N$ implies that $z_N^*$ and $z_0$ belong to the relative interior of the same cell in the normal manifold. Since the latter cell is a face of each $P_i$, $i=1,\cdots, k$, by the additional requirement \(z_N \in \mbox{int}P_i\) one has  \(z_N -z_N^*\in \cone(\mbox{int}P_i-z_N^*)\) and the latter set is exactly $\cone(\mbox{int}P_i-z_0)$, namely $\mbox{int}K_i$.

When \(k=1\), \(z_0\) is contained in the interior of an \(n\)-cell \(P_1\) and \(K_1=\mathbb{R}^n\). In this case \(Y^* \sim \mathcal{N}\left( 0,M_1^{-1}\Sigma_0M_1^{-T} \right)\), and (\ref{q:mainres}) follows from,
\[
\frac{\sqrt{N}(z_N-z_0)_j}{\eta^{\alpha}_j(\Sigma_N^{-1/2}\Phi_N(z_N),x_1)} \Rightarrow \frac{(Y^*)_j}{\eta^{\alpha}_j(\Sigma_0^{-1/2}d(f_0)_S(z_0),x_1)}.
\]

Next we consider the case when \(k \ge 2\). For all \(j=1,\dots,n\) and \(i=1,\dots,k\) let \(\bar{v}^{i,j}\in \Rset^n\) be such that \(\bar{v}^{i,j}\not\in K_i \) and \(|(\bar{v}^{i,j})_j|>\eta^{\alpha}_j(\Sigma_0^{-1/2}d(f_0)_S(z_0),x_i)\). Define random variables
\begin{eqnarray}
v^{i,j}_N&&=\sqrt{N}(z_N-z_0)\mathds{1}_{z_N\in B \cap \mbox{int}P_i}+\bar{v}^{i,j}\mathds{1}_{z_N \not\in B \cap \mbox{int}P_i}, \nonumber \\
\hat{Y}^{i,j}&&=Y^i\mathds{1}_{Y^i \in \mbox{int}K_i}+\bar{v}^{i,j}\mathds{1}_{Y^i \not\in \mbox{int}K_i}, \nonumber \\
\hat{\eta}^{i,j}_N&&=\eta^{\alpha}_j\left(\Sigma_N^{-1/2}\Phi_N(z_N),x_i \right)\mathds{1}_{z_N \in B \cap \mbox{int}P_i}+\eta^{\alpha}_j\left(\Sigma_0^{-1/2}d(f_0)_S(z_0),x_i\right)\mathds{1}_{z_N \not\in B \cap \mbox{int}P_i}, \nonumber
\end{eqnarray}
and note that
\[
\hat{\eta}^{i,j}_N\Rightarrow \eta^{\alpha}_j\left(\Sigma_0^{-1/2}d(f_0)_S(z_0),x_i\right).
\]
Next, for all Borel sets \(W\subset \mbox{int}K_i\),
\begin{eqnarray}
\Pr\left( v^{i,j}_N \in W \right) &&= \Pr\left(\sqrt{N}(z_N-z_0) \in W,\; z_N \in B\cap \mbox{int}P_i\right) \nonumber \\
&&=\Pr\left(\sqrt{N}(z_N-z_0) \in W,\; z_N \in B\right), \nonumber
\end{eqnarray}
and hence
\begin{eqnarray}
\lim\limits_{N\rightarrow \infty} \Pr \left(v^{i,j}_N \in W\right) &&= \lim\limits_{N\rightarrow \infty} \Pr \left(\sqrt{N}(z_N-z_0) \in W, z_N \in B\right) \nonumber\\
&&=\lim\limits_{N\rightarrow \infty}\Pr\left(\sqrt{N}(z_N-z_0) \in W\right) \nonumber \\
&&=\Pr\left(Y^* \in W\right)=\Pr\left(Y^i \in W\right)=\Pr\left(\hat{Y}^{i,j} \in W\right). \label{q:intcon}
\end{eqnarray}
Since \(z_N \rightarrow z_0\) in probability and $\mbox{int} K_i = \cone(\mbox{int} P_i -z_0)$, it follows that as \(N \rightarrow \infty\),
\[
\Pr \left( \sqrt{N} (z_N-z_0) \in (\mbox{int}K_i)^c, \; z_N \in B \cap \mbox{int}P_i \right) \rightarrow 0,
\]
and
\[
\Pr \left( z_N \not\in B \cap \mbox{int}P_i \right) \rightarrow
\Pr \left( Y^* \not\in \mbox{int}K_i \right) =
\Pr \left( Y^i \not\in \mbox{int}K_i \right)
=
\Pr \left( \hat{Y}^{i,j} \not\in \mbox{int}K_i \right).
\]
Thus for any Borel set \(D\) in \(\mathbb{R}^n\),
\begin{eqnarray}
&& \lim\limits_{N\rightarrow \infty}  \Pr \left( v_N^{i,j} \in D \cap (\mbox{int}K_i)^c \right) \nonumber\\
&& =\lim\limits_{N\rightarrow \infty} \mathds{1}_{D \cap (\mbox{int}K_i)^c}(\bar{v}^{i,j}) \Pr \left( z_N \not\in B \cap \mbox{int}P_i \right) \nonumber\\
&&  = \mathds{1}_{D \cap (\mbox{int}K_i)^c}(\bar{v}^{i,j}) \Pr \left( \hat{Y}^{i,j} \not\in \mbox{int}K_i \right) \nonumber \\
&& = \Pr \left( \hat{Y}^{i,j} \in D \cap (\mbox{int}K_i)^c \right). \label{q:cintcon}
\end{eqnarray}
Combining (\ref{q:intcon}) with (\ref{q:cintcon}) and since \(\hat{\eta}^{i,j}_N\) and \(\eta^{\alpha}_j\left(\Sigma_0^{-1/2}d(f_0)_S(z_0),x_i\right)\) are strictly positive under our assumptions we have that
\[
\frac{v^{i,j}_N}{\hat{\eta}^{i,j}_N} \Rightarrow \frac{\hat{Y}^{i,j}}{\eta^{\alpha}_j\left(\Sigma_0^{-1/2}d(f_0)_S(z_0),x_i\right)},
\]
and thus
\begin{eqnarray}
\lim\limits_{N\rightarrow \infty} && \Pr\left( \Big| \frac{(v^{i,j}_N)_j}{\hat{\eta}_N^{i,j}} \Big| \le 1\right) = \Pr\left( \Big| \frac{(\hat{Y}^{i,j})_j}{\eta^{\alpha}_j\left(\Sigma_0^{-1/2}d(f_0)_S(z_0),x_i\right)} \Big| \le 1\right) \nonumber \\
&&=\Pr\left( \Big| \frac{(Y^{i})_j}{\eta^{\alpha}_j\left(\Sigma_0^{-1/2}d(f_0)_S(z_0),x_i\right)} \Big| \le 1,\; Y^i \in \mbox{int}K_i\right), \nonumber
\end{eqnarray}
where we used the fact \( |(\bar{v}^{i,j})_j| > \eta^{\alpha}_j\left(\Sigma_0^{-1/2}d(f_0)_S(z_0),x_i\right)\). The latter fact also implies \(\lim\limits_{N\rightarrow \infty} \Pr\left( \Big| \frac{(\bar{v}^{i,j})_j}{\hat{\eta}^{i,j}_N} \Big| \le 1\right) = 0\), so it follows that
\begin{eqnarray}
&&\lim\limits_{N\rightarrow \infty} \Pr \left(\sqrt{N}\frac{|(z_N-z_0)_j|}{\eta^{\alpha}_j(\Sigma_N^{-1/2}\Phi_N(z_N),x_i)} \le 1;\;  A_N; \; z_N \in B \cap \mbox{int}P_i\right) \nonumber \\
&&=\lim\limits_{N\rightarrow \infty} \Pr \left(\sqrt{N}\frac{|(z_N-z_0)_j|}{\hat{\eta}^{i,j}_N} \le 1,\; z_N \in B \cap \mbox{int}P_i\right)=\lim\limits_{N\rightarrow \infty} \Pr \left(\frac{|(v^{i,j}_N)_j|}{\hat{\eta}^{i,j}_N} \le 1\right) \nonumber \\
&&=\Pr\left( \Big| \frac{(Y^{i})_j}{\eta^{\alpha}_j\left(\Sigma_0^{-1/2}d(f_0)_S(z_0),x_i\right)} \Big| \le 1, \;Y^i \in \mbox{int}K_i\right) \nonumber \\
&&=\Pr\left( |(M_i^{-1}\Sigma_0^{1/2}Z)_j| \le \eta^{\alpha}_j\left(\Sigma_0^{-1/2}d(f_0)_S(z_0),x_i\right),\; M_i^{-1}\Sigma_0^{1/2}Z \in K_i\right) \nonumber \\
&&=\Pr\left( |(d(f_0)_S^{-1}(z_0)\Sigma_0^{1/2}Z)_j| \le \eta^{\alpha}_j\left(\Sigma_0^{-1/2}d(f_0)_S(z_0),x_i\right),\; d(f_0)_S^{-1}(z_0)\Sigma_0^{1/2}Z \in K_i\right) \nonumber\\
&&=(1-\alpha)\Pr\left( d(f_0)_S^{-1}(z_0)\Sigma_0^{1/2}Z \in K_i\right).\nonumber
\end{eqnarray}
Finally, since on \(A_N\) we have \(z_N-z_N^* \in \mbox{int}K_i\),
\begin{eqnarray}
&& \lim\limits_{N\rightarrow \infty} \Pr \left(\sqrt{N}|(z_N-z_0)_j| \le \eta^{\alpha}_j(\Sigma_N^{-1/2}\Phi_N(z_N),z_N-z_N^*) \right) \nonumber \\
&&=\lim\limits_{N\rightarrow \infty}\sum\limits_{i=1}^{k} \Pr \left(\sqrt{N}\frac{|(z_N-z_0)_j|}{\eta^{\alpha}_j(\Sigma_N^{-1/2}\Phi_N(z_N),x_i)} \le 1;\;  A_N; \; z_N \in B \cap \mbox{int}P_i\right) \nonumber \\
&&=\sum\limits_{i=1}^{k}(1-\alpha)\Pr\left( d(f_0)_S^{-1}(z_0)\Sigma_0^{1/2}Z \in K_i\right) \nonumber \\
&&=(1-\alpha)\sum\limits_{i=1}^{k}\Pr\left( d(f_0)_S^{-1}(z_0)\Sigma_0^{1/2}Z \in K_i\right)=1-\alpha.\nonumber
\end{eqnarray}
\end{proof} \qed

Comparing the above two methods, computation of $\eta^{\alpha}_j(\Sigma_N^{-1/2}\Phi_N(z_N),z_N-z_N^*)$ is more efficient than that of $a^r\left( ( \Phi_N^{-1}(z_N)\Sigma_N^{1/ 2} )_j \right)$, as it with high probability restricts the computation to a single cone in the conical subdivision of $\Phi_N(z_N)$, namely the cone that contains $z_N-z_N^*$ in its interior (the same cone also contains $z_N-z_0$ in its interior whenever the event $A_N$ in (35) holds).


\subsection{The third method (a direct approach)}\label{s:icim1}

Comparing the asymptotic distributions for $z_N$ and $x_N$, as given by (\ref{q:zN_dist_cov}) and (\ref{q:xN_dist_cov}) respectively, we see that the latter distribution has \(\Pi_{K_0}\) in it, the projector onto the critical cone to \(S\) at \(z_0\). Since \(\Pi_{K_0}\) is generally non-invertible, neither of the methods presented in \S \ref{s:nmicim1} and \S \ref{s:nmicim2} can be used to directly construct intervals for \( (x_0)_j\). Both methods require the invertibility of the function appearing in the asymptotic distribution either in the construction of an interval or the proof of the interval's exactness.

The non-invertibility of \(\Pi_{K_0}\) also leads us to change our focus from asymptotically exact intervals to intervals meeting a specified minimum level of confidence for the following reason. If the function \(\Pi_{K_0} \circ (L_{K_0})^{-1}(\cdot)\) appearing in (\ref{q:xN_dist_cov}) has a selection function whose matrix representation contains a row of zeros (say the $j$th row), then there is a non-zero probability for \((x_N)_j\) to equal \( (x_0)_j \). In this case any reasonable method for constructing individual confidence intervals of $(x_0)_j$ will have a lower bound on its performance:  no matter how narrow the interval is, the probability for it to contain $(x_0)_j$ is no less than the probability for $(x_0)_j$ and $(x_N)_j$ to coincide.

The method to be presented below determines the interval width based on equation (\ref{q:icimres}) in Theorem \ref{t:icim}, in which \(h^{\alpha}_j(\cdot,\cdot,\cdot)\) replaces \(\eta^{\alpha}_j(\cdot,\cdot)\) in the previous method. Below we introduce the definition of \(h^{\alpha}_j(f,g,x)\), where \(f\) and \(g\) are piecewise linear functions from \(\mathbb{R}^n\) to \(\mathbb{R}^n\) that share a common conical subdivision, \(\left\{ K_1,\dots,K_k\right\}\), with \(g\) invertible. For any choice of cone \(K_i\), \(i=1,\dots,k\), component \(j =1,\dots,n\) and \(\alpha \in (0,1)\) we first define \(h^{\alpha}_j(f,g,x)\) for points \(x \in \mbox{int}K_i\) to be
\[
h^{\alpha}_j(f,g,x)=\inf \left\{ l \ge 0 \; \Big| \; \frac{ \Pr \left( |\left(f(g^{-1}(Z))\right)_j| \le l \mbox{ and } g^{-1}(Z) \in K_i \right)}{\Pr\left( g^{-1}(Z) \in K_i \right)} \ge (1-\alpha) \right\}.
\]
Denoting the matrix representations of the selection functions on each cone as \(f|_{K_i}=Q_i\) and \(g|_{K_i}=M_i\), for all points \(x \in \mbox{int}K_i\) the function \(h^{\alpha}_j(f,g,x)\) will take the same value and the above definition is equivalent to
\begin{equation}\label{q:hici}
h^{\alpha}_j(f,g,x)=\inf \left\{ l \ge 0 \; \Big| \; \frac{\Pr \left( |(Q_i)_jM_i^{-1}Z| \le l \mbox{ and } M_i^{-1}Z \in K_i \right)}{\Pr\left( M_i^{-1}Z \in K_i \right) } \ge (1-\alpha)\right\}.
\end{equation}
For points \(x \in \bigcap_{s=1}^v K_{i_s}\) define \(h^{\alpha}_j(f,g,x)=\max_{s=1,\dots,v} h^{\alpha}_j(f,g,x_{i_s})\) where \(x_{i_s} \in \mbox{int}K_{i_s}\). As shown in the following lemma we can identify when \(h^{\alpha}_j(f,g,x)=0\) based on \(x\) and the matrix representations for the appropriate selection functions of \(f\).
\begin{lemma}\label{t:icihz}
For any point \(x \in \bigcap_{s=1}^v K_{i_s}\), \(j=1,\dots,n\) and \(\alpha \in (0,1)\), \(h^{\alpha}_j(f,g,x)=0\) if and only if \( (Q_{i_s})_j\) is the zero vector for all \(s=1,\dots,v\).
\end{lemma}
\begin{proof}
It suffices to prove the result for \(x \in \mbox{int} K_i\). If \(h^{\alpha}_j(f,g,x)=0\),
\[
0< (1-\alpha)\Pr\left( M^{-1}_iZ \in K_i \right) \le \Pr \left( | (Q_i)_jM^{-1}_iZ| \le 0 \mbox{ and } M^{-1}_iZ \in K_i \right),
\]
and hence,
\begin{equation}\label{q:icihz}
0 < \Pr \left( (Q_i)_jM^{-1}_iZ=0 \mbox{ and } M^{-1}_iZ \in K_i \right) \le \Pr \left( (Q_i)_jM^{-1}_iZ=0 \right).
\end{equation}
Since \((Q_i)_jM^{-1}_iZ\sim \mathcal{N}\left(0,\|(Q_i)_jM^{-1}_i\|^2\right)\), where \(\|\cdot\|\) denotes the Euclidian norm, (\ref{q:icihz}) implies that \(\|(Q_i)_jM^{-1}_i\|=0\), and thus \((Q_i)_j\) is a vector of zeroes. The reverse implication follows immediately.
\end{proof} \qed

When using \(h^{\alpha}_j(f,g,x)\) to construct confidence intervals for solutions to (\ref{q:true}) we will be interested in
\[
f=\Pi_{K_0} \mbox{ and }\; g=\Sigma_0^{-1/2}d(f_0)_S(z_0)
\]
and their estimates
\[
f_N=\Lambda_N(z_N)=d\Pi_{S}(z_N^*) \mbox{ and } \; g_N=\Sigma_N^{-1/2}\Phi_N(z_N).
\]
From  (\ref{q:dSnK}) and (\ref{q:bdcon}) it follows that the probability of all four functions sharing a common conical subdivision and \(f_N\) equalling \(f\) goes to one as the sample size goes to infinity. We therefore take this to be the setting for the following lemma.
\begin{lemma}\label{t:iciconv}
Let \(f, g:\mathbb{R}^n \rightarrow \mathbb{R}^n\) be piecewise linear functions with \(g\) a homeomorphism. Suppose that \(\left\{f_N\right\}_{N=1}^{\infty}\) and \(\left\{g_N\right\}_{N=1}^{\infty}\) are two sequences of piecewise linear functions such for all \(N\) sufficiently large
\begin{enumerate}
\item \(f_N=f\).
\item \(f\), \(g\) and \(g_N\) all share a common conical subdivision \(\left\{K_1,\dots,K_k\right\}\).
\item \( \sup\limits_{h \in \mathbb{R}^n,h \ne 0} \frac{\| g_N(h)-g(h)\| }{\|h\|} \rightarrow 0.\)
\end{enumerate}
Then for all \(N\) sufficiently large \(g_N\) will be a homeomorphism and \(h^{\alpha}_j(f_N,g_N,x) \rightarrow h^{\alpha}_j(f,g,x)\) for all \(x \in \mathbb{R}^n\), \( \alpha \in (0,1)\) and \(j=1,\dots,n\).
\end{lemma}
\begin{proof}
From Proposition \ref{t:convprop} it follows that for all \(N\) sufficiently large \(g_N\) is a homeomorphism and that \(g_N^{-1}\) converges uniformly on compacts to \(g^{-1}\). Next take \(v\) to be large enough so that for all \(N \ge v\) the functions \(g_N\) are invertible, \(f_N=f\) and \(f\), \(g\) and \(g_N\) all share common conical subdivision \(\left\{K_1,\dots,K_k\right\}\). To prove the remainder of the Lemma's claim it suffices to show that for any \(x \in \mbox{int}K_i\), \(i=1,\dots,k\), \(h^{\alpha}_j(f,g_N,x) \rightarrow h^{\alpha}_j(f,g,x)\).

When \(x \in \mbox{int}K_i\) and \(h^{\alpha}_j(f,g,x)=0\), it follows from Lemma \ref{t:icihz} that \(h^{\alpha}_j(f,g_N,x)=0\). In the case of \(x \in \mbox{int}K_i\) and \(h^{\alpha}_j(f,g,x) > 0\), the convergence can be shown using an argument analogous to the proof of Lemma \ref{t:lem1} and Lemma \ref{t:m2dc} and is therefore omitted.

\end{proof} \qed

The main result of this section, Theorem \ref{t:icim}, can now be proven. We will use the same notation used in Theorem \ref{t:meth2} where \(\Gamma'(z_0)=\left\{K_1,\dots,K_k\right\}\) is the conical subdivision associated with \(d(f_0)_S(z_0)\) such that \(d(f_0)_S(z_0)|_{K_i}=M_i\) and \(K_i=\cone(P_i-z_0)\), where \(P_1,\dots,P_k\) are all \(n\)-cells in the normal manifold of \(S\) that contain \(z_0\). We additionally denote \(\Pi_{K_0}|_{K_i}=Q_i\) and define the following random variables:
\[
\begin{array}{c c c c}
Y^i=M_i^{-1}\Sigma_0^{1/2}Z, & Y_0=\Sigma_0^{1/2}Z & \mbox{and} & Y^*=d(f_0)_S^{-1}(z_0)\Sigma_0^{1/2}Z.
\end{array}
\]
\begin{theorem}\label{t:icim}
Let Assumptions 2, 3 and 4 hold. Let \(\Phi_N(z_N)(\cdot)\) and \(z_N^*\) be as defined in (\ref{q:funcent}). For all \(j=1,\dots,n\) and \(\alpha \in (0,1)\),
\begin{equation}\label{q:icimres}
\lim\limits_{N\rightarrow \infty}\Pr\left(\sqrt{N}|(x_N-x_0)_j| \le h^{\alpha}_j(d\Pi_S(z_N^*),\Sigma_N^{-1/2}\Phi_N(z_N),z_N-z_N^*) \right) \ge 1-\alpha.
\end{equation}
\end{theorem}
\begin{proof}
As in the proof of Theorem \ref{t:meth2} we begin by letting \(C_i\), \(i=1,\dots,l\) denote the cells in the normal manifold of \(S\) and for each \(N\) let the event \(A_N\) be as defined in (\ref{q:consubcon}). Now for \(\omega \in A_N\) the equality \(\Pi_{K_0}=d\Pi_S(z_N^*)\) holds, and \( \left\{K_1,\dots,K_k \right\}\) provides a common conical subdivision for  \(\Pi_{K_0}\), \(d(f_0)_S(z_0)\) and \(\Phi_N(z_N(\omega))\). From (\ref{q:tconeq}) and Lemma \ref{t:iciconv} it follows that for all fixed \(u\), \(h^{\alpha}_j(d\Pi_S(z_N^*),\Sigma_N^{-1/2}\Phi_N(z_N),u)\) converges in probability to \(h^{\alpha}_j(\Pi_{K_0},\Sigma_0^{-1/2}d(f_0)_S(z_0),u)\).

Next let \(B\) be a fixed neighborhood of \(z_0\) such that \(B\cap( z_0+ K_i) = B \cap P_i\) for \(i=1,\dots,k\). We then have
\begin{eqnarray}
&& \lim\limits_{N\rightarrow \infty} \Pr \left(\sqrt{N}|(x_N-x_0)_j| \le h^{\alpha}_j(d\Pi_S(z_N^*),\Sigma_N^{-1/2}\Phi_N(z_N),z_N-z_N^*) \right) \nonumber \\
=&& \lim\limits_{N\rightarrow \infty} \Pr \left(\sqrt{N}|\left(\Pi_S(z_N)-\Pi_S(z_0)\right)_j| \le h^{\alpha}_j(d\Pi_S(z_N^*),\Sigma_N^{-1/2}\Phi_N(z_N),z_N-z_N^*); \; A_N\right) \nonumber \\
=&& \lim\limits_{N\rightarrow \infty} \sum\limits_{i=1}^k\Pr \left(\sqrt{N}|\left(\Pi_{K_0}(z_N-z_0)\right)_j| \le h^{\alpha}_j(d\Pi_S(z_N^*),\Sigma_N^{-1/2}\Phi_N(z_N),z_N-z_N^*); \; A_N; \; z_N \in B \cap \mbox{int}P_i\right) \nonumber \\
=&& \sum\limits_{i=1}^k \lim\limits_{N\rightarrow \infty} \Pr \left(\sqrt{N}|(Q_i)_j(z_N-z_0)| \le h^{\alpha}_j(d\Pi_S(z_N^*),\Sigma_N^{-1/2}\Phi_N(z_N),u_i); \; A_N; \; z_N \in B \cap \mbox{int}P_i\right) \label{q:icisum}
\end{eqnarray}
where \(u_i\) is any point in \(\mbox{int}K_i\). The first equality uses the relation between solutions to a variational inequality and its normal map formulation, while the second equality combines the almost sure convergence of \(z_N\) to \(z_0\) with (\ref{q:SnK}). The final equality uses the fact that for \(\omega \in A_N\) and \(z_N \in \mbox{int}P_i\) both \(z_N-z_0\) and \(z_N-z_N^*\) will be contained in \(\mbox{int}K_i\) and thus \(z_N-z_N^*\) may be replaced with \(u_i\) and \(\Pi_{K_0}(z_N-z_0)=Q_i(z_N-z_0)\).

 Evaluating each term in (\ref{q:icisum}) depends on \((Q_i)_j\). If \((Q_i)_j\) is the zero vector for some $i$, then
\begin{eqnarray}
&&\lim\limits_{N\rightarrow \infty}\Pr \left(\sqrt{N}|(Q_i)_j(z_N-z_0)| \le h^{\alpha}_j(d\Pi_S(z_N^*),\Sigma_N^{-1/2}\Phi_N(z_N),u_i); \; A_N; \; z_N \in B \cap \mbox{int}P_i\right) \nonumber \\
&&= \lim\limits_{N\rightarrow \infty}\Pr \left( \sqrt{N}(z_N-z_0) \in \mbox{int}K_i \right)=\Pr \left( Y^* \in \mbox{int}K_i \right) \nonumber \\
&&= \Pr \left( d(f_0)^{-1}_S(z_0)\Sigma_0^{1/2}Z \in K_i\right). \label{q:icizp}
\end{eqnarray}

On the other hand, if \((Q_i)_j\) is a nonzero vector (i.e., it contains at least one nonzero element) for some $i$, we define a vector \(\bar{v}^{i,j}\) to be such that \(\bar{v}^{i,j} \not\in K_i\) and \( | (Q_i)_j\bar{v}^{i,j}| > h^{\alpha}_j(\Pi_{K_0},\Sigma_0^{-1/2}d(f_0)_S(z_0),u_i)\). With these we define random vectors
\begin{eqnarray}
v^{i,j}_N&&=\sqrt{N}(z_N-z_0)\mathds{1}_{z_N\in B \cap \mbox{int}P_i}+\bar{v}^{i,j}\mathds{1}_{z_N \not\in B \cap \mbox{int}P_i}, \nonumber \\
\hat{Y}^{i,j}&&=Y^i\mathds{1}_{Y^i \in \mbox{int}K_i}+\bar{v}^{i,j}\mathds{1}_{Y^i \not\in \mbox{int}K_i}, \nonumber \\
\hat{h}^{i,j}_N&&=h^{\alpha}_j\left(d\Pi_S(z_N^*),\Sigma_N^{-1/2}\Phi_N(z_N),u_i \right)\mathds{1}_{z_N \in B \cap \mbox{int}P_i}+h^{\alpha}_j\left(\Pi_{K_0},\Sigma_0^{-1/2}d(f_0)_S(z_0),u_i\right)\mathds{1}_{z_N \not\in B \cap \mbox{int}P_i}. \nonumber
\end{eqnarray}
Using the same arguments as in Theorem \ref{t:meth2} it follows that
\[
\frac{v_N^{i,j}}{\hat{h}^{i,j}_N} \Rightarrow \frac{\hat{Y}^{i,j}}{h^{\alpha}_j\left(\Pi_{K_0},\Sigma_0^{-1/2}d(f_0)_S(z_0),u_i\right)}
\]
and
\begin{eqnarray}
&&\lim\limits_{N\rightarrow \infty} \Pr \left(\sqrt{N}\frac{|(Q_i)_j(z_N-z_0)|}{h^{\alpha}_j(d\Pi_S(z_N^*),\Sigma_N^{-1/2}\Phi_N(z_N),u_i)} \le 1;\; A_N; \; z_N \in B \cap \mbox{int}P_i\right) \nonumber \\
&&=(1-\alpha)\Pr\left( d(f_0)_S^{-1}(z_0)\Sigma_0^{1/2}Z \in K_i\right).\label{q:icinzp}
\end{eqnarray}

Combining (\ref{q:icizp}) and (\ref{q:icinzp}), with the fact that \(z_N-z_N^* \in \mbox{int}K_i\) on \(A_N\), we have
\begin{eqnarray}
&& \lim\limits_{N\rightarrow \infty} \Pr \left(\sqrt{N}|(x_N-x_0)_j| \le h^{\alpha}_j(d\Pi_S(z_N^*),\Sigma_N^{-1/2}\Phi_N(z_N),z_N-z_N^*) \right) \nonumber \\
&&=\lim\limits_{N\rightarrow \infty}\sum\limits_{i=1}^{k} \Pr \left(\sqrt{N}|(Q_i)_j(z_N-z_0)| \le h^{\alpha}_j(d\Pi_S(z_N^*),\Sigma_N^{-1/2}\Phi_N(z_N),u_i);\; A_N; \; z_N \in B \cap \mbox{int}P_i\right) \nonumber \\
&& \ge (1-\alpha)\sum\limits_{i=1}^{k}\Pr\left( d(f_0)_S^{-1}(z_0)\Sigma_0^{1/2}Z \in K_i\right)=1-\alpha.\nonumber
\end{eqnarray}
\end{proof} \qed

An important fact seen in the proof of Theorem \ref{t:icim} is that
\[
\lim\limits_{N\rightarrow \infty}\Pr\left(\sqrt{N}|(x_N-x_0)_j| \le h^{\alpha}_j(\Pi_{K_N},\Sigma_N^{-1/2}\Phi_N(z_N),z_N-z_N^*) \right) > 1-\alpha
\]
if and only if there exists a cone \(K_i\) in the conical subdivision of \(\Pi_{K_0}\) such that the \(j^{\tiny\mbox{th}}\) component of \(\Pi_{K_0}|_{K_i}\) is zero. When this is the case, we have
\[
(x_N(\omega)-x_0)_j=\big(\Pi_S(z_N(\omega))-\Pi_S(x_0)\big)_j=(Q_i)_j\left( z_N(\omega)-z_0\right)=0,
\]
as long as \(z_N(\omega)-z_0 \in \mbox{int}K_i\) and $z_N$ is sufficiently close to $z_0$.
If additionally \(\omega \in A_N\), then \(d\Pi_S(z_N^*)=\Pi_{K_0}\) and   we have by Lemma \ref{t:icihz}
 \[h^{\alpha}_j(d\Pi_S(z_N^*),\Sigma_N^{-1/2}\Phi_N(z_N),z_N-z_N^*) =0, \]
 meaning that the method of Theorem \ref{t:icim} returns the correct point estimate \(\left(x_N(\omega)\right)_j=\left(x_0\right)_j\). Recalling that \(\lim_{N\rightarrow \infty}\Pr(A_N)=1\) and that $z_N$ converges to $z_0$ almost surely, we see that the potentially conservative asymptotic level of confidence is not the result of using unnecessarily long intervals, but instead follows from the fact that for sufficiently large sample sizes the proposed method will return the correct point estimate with a nonzero probability.

While (\ref{q:result}), (\ref{q:mainres}) and (\ref{q:icimres}) provide computable intervals with the desired asymptotic properties, in general \(a^r(\cdot)\), \(\eta^{\alpha}_j(\cdot,\cdot)\) and \(h_J^{\alpha}(\cdot,\cdot,\cdot)\) lack closed form expressions. In the next section we consider the computation of these quantities. For ease of exposition, moving forward we will suppress the arguments of \(a^{r}\), \(\eta^{\alpha}_j\) and \(h_j^{\alpha}\).

\section{Interval Computation}\label{s:comp}

This section considers the computation of \(a^{r}\), \(\eta^{\alpha}_j\) and \(h_J^{\alpha}\), and discusses how to find upper bounds for these quantities. Before presenting a general method for computing \(a^{r}\), \(\eta^{\alpha}_j\) and \(h_J^{\alpha}\) we consider special cases when either closed form expressions exist or less burdensome techniques can be used. For each of these discussions we begin by considering \(a^r\) with the results for  \(\eta^{\alpha}_j\) and \(h_J^{\alpha}\) following in a similar fashion.

The first case we consider is when \(\Sigma_N^{-1/2}\Phi_N(z_N)\) and \(d\Pi_S(z_N^*)\) are linear functions with matrix representations \(M_N\) and \(Q_N\) respectively. Since $(M_N^{-1} Z)_j$ is a mean zero Normal random variable for each coordinate \(j\), it is natural to set \(r=0\) for \(a^r\). Then from basic properties of Normal random vectors,
\[a^0=\eta^{\alpha}_j=\sqrt{\chi^2_1(\alpha)\| (M_N^{-1})_j\|^2} \; \; \mbox{ and } \; \; h^{\alpha}_j=\sqrt{\chi^2_1(\alpha)\| (Q_N)_jM_N^{-1}\|^2}\]
where \(\| \cdot \|\) is the Euclidian norm. Note in this case both intervals for \((z_0)_j\) are the same as the interval considered in Theorem \ref{t:indold}.

In the piecewise linear case let \(\phi_{N,j}\) denote the \(j^{\tiny\mbox{th}}\) component function of \(\Phi_N^{-1}(z_N)\Sigma_N^{1/2}\). Finding \(a^r\) requires a search over values of \(l > 0\) and evaluating \(\Pr \left( |\phi_{N,j}(Z)-r| \le l \right)\). To evaluate this probability we rewrite it in terms of the selection functions of \(\phi_{N,j}\).

To this end, let \(\Gamma=\left\{K_1,\dots,K_k\right\}\) be the common conical subdivision for \(\Sigma_N^{-1/2}\Phi_N(z_N)\) and \(d\Pi_S(z_N^*)\), and let \(\left\{M_{N,1},\dots,M_{N,k}\right\}\) and \(\left\{Q_{N,1},\dots,Q_{N,k}\right\}\) be the matrix representations for the respective selection functions. Then with
\[
T_i=\Sigma_N^{-1/2}\Phi_N(z_N)(K_i)=M_{N,i}(K_i),
\]
\(\left\{T_1,\dots,T_k\right\}\) provides a conical subdivision for \(\phi_{N,j}\) such that  \(\phi_{N,j}|_{T_i}=(M_{N,i}^{-1})_j\). Due to the high probability of \(\Sigma_N^{-1/2}\Phi_N(z_N)\) and \(d(f_0)_S(z_0)\) sharing a common conical subdivision we have used the same notation \(K_i\), \(i=1,\dots,k\), as was introduced before Theorem \ref{t:indold}. In the discussion that follows it is not necessary for the functions to share a common conical subdivision. Additionally, any assumptions made about the value of \(k\) will refer to the number of selection functions for a particular realization of \(\Sigma_N^{-1/2}\Phi_N(z_N)\) unless otherwise stated.

 For any two cones \(T_v,T_u \in \Gamma'\) with \(v \ne u\), their intersection is either empty or a proper face of both cones, and hence \(\Pr\left( Z \in T_v\cap T_u\right) =0\). The probability we need to evaluate can thus be rewritten as
\begin{eqnarray}
\Pr \left( |\phi_{N,j}(Z)-r| \le l \right)&&=\sum\limits_{i=1}^k \Pr \left( |\phi_{N,j}(Z)-r| \le l \mbox{ and } Z \in T_i\right) \nonumber \\
&& =\sum\limits_{i=1}^k \Pr \left( |(M_{N,i}^{-1})_jZ-r| \le l \mbox{ and } Z \in T_i\right) \label{q:poi}.
\end{eqnarray}

Note the connection between (\ref{q:poi}) and what must be considered to find \(\eta^{\alpha}_j\). Finding \(\eta^{\alpha}_j\) requires us to evaluate
\begin{equation}\label{q:m2poi}
\Pr \left( |(M_{N,i}^{-1})_jZ| \le l \mbox{ and }M_{N,i}^{-1}Z \in K_i \right)=\Pr \left( |(M_{N,i}^{-1})_jZ| \le l \mbox{ and } Z \in T_i \right),
\end{equation}
for different values of \(l\), but only for those indices \(i\)  such that \(z_N-z_N^* \in K_i\). At this point we see the computational benefits of \(\eta^{\alpha}_j\) over \(a^r\). Recall from the proof of Theorem \ref{t:meth2} that
\[
\lim_{N\rightarrow \infty}\sum_{i=1}^k \Pr \left( A_N \mbox{ and } z_N \in B \cap \mbox{int}  P_i\right) =1,
\]
where \(k\) is the number of selection functions for \(d(f_0)_S(z_0)\), \(A_N\) is as defined in (\ref{q:consubcon}) and \(K_i=\cone(P_i-z_0)\) .  Moreover when \(A_N\) holds and \(z_N \in B \cap \mbox{int} P_i\) it was argued that \(z_N-z_N^* \in \mbox{int} K_i\). Therefore with high probability each value of \(l\) we consider when finding \(\eta^{\alpha}_j\) will  involve evaluating (\ref{q:m2poi}) for a single index \(i\). In contrast, (\ref{q:poi}) involves a similar calculation for every cone in the subdivision. Finding \(h^{\alpha}_j\) will with high probability also require considering only a single index \(i\), but with the quantity evaluated being \( \Pr \left( |(Q_{N,i})_jM_{N,i}^{-1}Z| \le l \mbox{ and } Z \in T_i \right)\).

The question of finding \(a^r\), \(\eta^{\alpha}_j\) and \(h^{\alpha}_j\) in the piecewise linear case now becomes how to evaluate
\begin{equation}\label{q:bmpoi}
\Pr \left( |b_{N,i}^TZ-r| \le l \mbox{ and } Z \in T_i\right),
\end{equation}
where \(b_{N,i}^T=(M_{N,i}^{-1})_j\) when finding \(a^r\) and \(\eta^{\alpha}_j\), and \(b_{N,i}^T=(Q_{N,i})_jM_{N,i}^{-1}\) when finding \(h^{\alpha}_j\). When \(k=2\) and \(r=0\) evaluating (\ref{q:bmpoi}) is simplified by observing that the two cones in \(\Gamma'\) satisfy \(T_1=-T_2\) and the fact that \(Z\) and \(-Z\) have the same distribution. It then follows that
\begin{eqnarray}
\Pr \left( |b_{N,i}^TZ| \le l \mbox{ and } Z \in T_i \right)&&= 1/2 \Pr \left( |b_{N,i}^TZ| \le l \right)\nonumber \\
&&=\Pr\left(Z\in T_i\right) \Pr \left( |b_{N,i}^TZ| \le l \right). \nonumber
\end{eqnarray}
In this case no search is necessary for \(\eta^{\alpha}_j\) and \(h^{\alpha}_j\). Finding \(a^0\) may still require a search over different values of \(l\) but this search can refer to the cumulative distribution function of a standard Normal random variable to evaluate the necessary probabilities.

When \(k>2\) our approach to evaluating (\ref{q:bmpoi}) is to rewrite it as the probability of a Normal random vector being in a box with possibly infinite endpoints. Once formulated in this manner the probability can be evaluated either using the numerical techniques of \cite{metal:oprob} or the Monte Carlo and Quasi-Monte Carlo methods of \cite[Chapter 4]{g.b:np}, both of which are implemented in R package mvtnorm \cite{g.b:np,mvtnorm}. Comparisons of the methods for different problem sizes can be found in \cite{np:comp}. The method in \cite{metal:oprob} requires the Normal random vector of interest to have a non-singular covariance matrix, so we first consider a class of SVIs for which this condition holds when finding \(a^r\) and \(\eta^{\alpha}_j\).

When the SVI is a complementarity problem with \(S=\mathbb{R}^m \times \mathbb{R}_+^{n-m}\), where \(\mathbb{R}^k_+\) denotes the positive orthant, each of the polyhedral cones \(K_i \in \Gamma\) can be expressed as an \(n\)-dimensional box,
\[
K_i= [l_1^i,u_1^i]\times  \dots \times [l_n^i,u_n^i]
\]
with \(0\), \(\infty\) or \(-\infty\) as endpoints. Additionally by \(\Phi^{-1}_N(z_N)\Sigma_N^{1/2}\) a homeomorphism it follows that for each \(i=1,\dots,k\) and \(x\in \mathbb{R}^n\)
\[
x \in T_i  \Leftrightarrow \Phi^{-1}_N(z_N)\Sigma_N^{1/2}(x) \in K_i \Leftrightarrow M_{N,i}^{-1}x \in K_i.
\]
Therefore we can write
\begin{eqnarray}
&&\Pr \left(|(M_{N,i}^{-1})_jZ-r| \le l \mbox{ and } Z \in T_i \right) = \Pr \left(r-l \le (M_{N,i}^{-1})_jZ \le r+l \mbox{ and } M_{N,i}^{-1}Z \in K_i\right) \nonumber\\
&&=\Pr\left( M_{N,i}^{-1}Z \in [l_1^i,u_1^i]\times \dots \times [ \max(l_j^i,r-l),\min(u_j^i,r+l)]\times \dots \times [l_n^i,u_n^i] \right )\nonumber \\
&&= \Pr\left( \tilde{Z} \in [l_1^i,u_1^i]\times \dots \times [\max(l_j^i,r-l),\min(u_j^i,r+l)]\times \dots \times [l_n^i,u_n^i] \right) \nonumber
\end{eqnarray}
where \(\tilde{Z} \sim \mathcal{N}\left(0,M_{N,i}^{-1}M_{N,i}^{-T}\right)\). It follows that \( \tilde{Z}\) has a non-singular covariance matrix, and either method of evaluating the probability can be used. Note that this approach cannot be used to find \(h^{\alpha}_j\) due to the additional consideration of \( (Q_{N,i})_j\).

In general to compute \(a^r\), \(\eta^{\alpha}_j\) and \(h^{\alpha}_j\) we can use the structure of \(T_i\) being a polyhedral cone. In this case we express the cone as a system of linear inequalities,
\[
T_i=\left\{ x \in \mathbb{R}^n | C_ix \le 0_v \right\}
\]
for \(C_i\) some \(v\times n\) matrix and \(0_v\) the \(v\)-dimensional zero vector. We then rewrite
\begin{eqnarray}
&&\Pr\left( | b_{N,i}^TZ-r|\le l \mbox{ and  } C_iZ \le 0_v \right) =\Pr \left( \bar{Z} \in (-\infty,0]\times \dots \times (-\infty,0]\times [r-l, r+l] \right) \nonumber \\
&&\mbox{ where } \bar{Z} \sim \mathcal{N}\left(0_{v+1},D_iD_i^T\right)  \mbox{   and   } D_i=\left[\begin{array}{c} C_i \\b_{N,i}^T \end{array} \right]. \nonumber
\end{eqnarray}
When the covariance matrix of \(\bar{Z}\) is singular only the methods of \cite[Chapter 4]{g.b:np} may be employed.

The potential of having to search over values of \(l\) when finding \(a^r\) leads us to consider the question finding an upper bound for \(a^r\).  Since for linear functions \(a^r\) is easily found one might hope that for \(\psi\) piecewise linear with family of selection functions expressed as \(n\) dimensional row vectors \( \left\{ b^1,\dots,b^k\right\}\), with \( \|b^1\|\le \|b^2\| \le \dots \le \|b^k\|\), that \( a^r(\psi) \le a^r(b^k)\). This need not be true.

For example take
\[
b^1=\left[\begin{array}{c c}1/5 & 7/5 \end{array} \right], b^2=\left[\begin{array}{c c}7/5 & 1/5\end{array} \right], b^3= \left[\begin{array}{c c}1 & 1\end{array} \right],
\]
and \(\gamma_i= \left\{x\in \Rset^2 | C_ix \le 0\right\}\) for \(i=1,\dots,5\), where
\begin{eqnarray}
 C_1=\left[\begin{array}{c c}1 & -1 \\ 2 & -1 \end{array} \right],&& C_2=\left[\begin{array}{c c}-1 & 1 \\ -1 & 2 \end{array} \right], C_3=\left[\begin{array}{c c}  -2 & 1 \\ 1 & -2 \end{array} \right] , \nonumber \\
 C_4=\left[\begin{array}{c c} 1 & -1 \end{array} \right]  && \mbox{ and  }  C_5=\left[\begin{array}{c c} -1 & 1 \end{array} \right]. \nonumber
\end{eqnarray}
Note both \(\left\{ \gamma_1,\gamma_2,\gamma_3\right\}\) and \(\left\{\gamma_4, \gamma_5\right\}\) are conical subdivisions of \(\mathbb{R}^2\). Define \(\psi_1\) and \(\psi_2\) to be piecewise linear functions such that \(\psi_1|_{\gamma_i}=b^i\)  for \(i=1,2,3\), \(\psi_2|_{\gamma_4}=b^1\) and \(\psi_2|_{\gamma_5}=b^2\) . It follows that \( a^0(b^i)=a^0(\psi_2)=\sqrt{2\chi^2_1(\alpha)},\) \(i=1,2,3\). Next let
\begin{eqnarray}
R_1&&=\left\{ z \in \mathbb{R}^2 | -a^0(\psi_2) \le \psi_1(z) \le a^0(\psi_2)\right\}, \nonumber\\
R_2&&= \left\{ z \in \mathbb{R}^2 | -a^0(\psi_2) \le \psi_2(z) \le a^0(\psi_2)\right\} \nonumber.
\end{eqnarray}

As shown in Figure \ref{f:regfig}, the set \(R_2\) includes \(R_1\) as a subset with \(D=R_2\setminus R_1\) having a non-empty interior. Thus \( \Pr \left( Z \in R_1 \right) < \Pr \left( Z \in R_2 \right) \) and \(a^0(\psi_2) < a^0(\psi_1)\), showing that \(\max a^0(b^i)\) is not an upper bound for \(a^0(\psi_1)\).

\begin{figure}[h!]
\centering
\includegraphics[scale=.75]{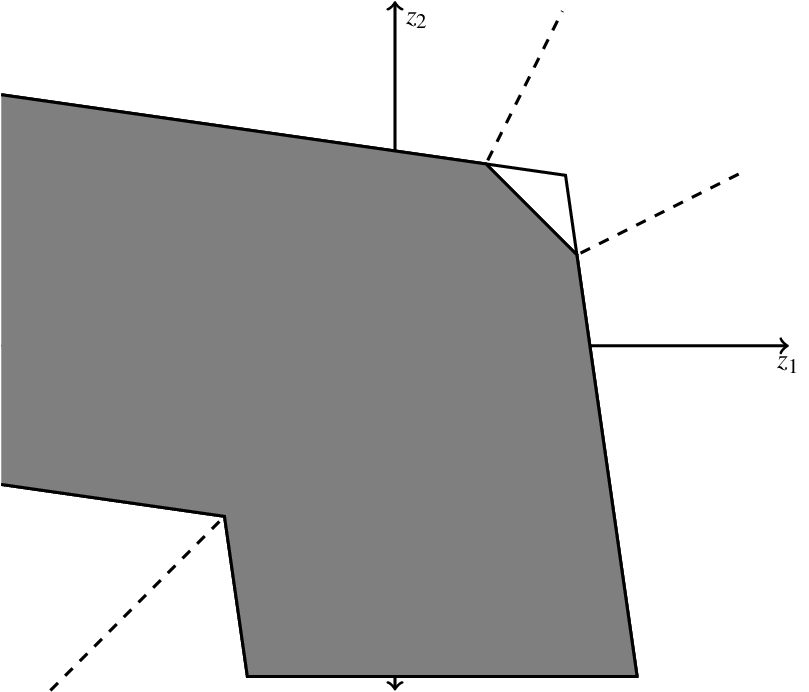}
\caption{Sets \(R_1\) (shaded) and \(R_2\) for \(\alpha
=.05\)}\label{f:regfig}
\end{figure}

To construct a valid upper bound for an interval's half width we will use the following Lemma.

\begin{lemma}\label{t:blem}
Let \(f:\mathbb{R}^n \rightarrow \mathbb{R}\) be a piecewise linear function with family of selection functions given by \(n\) dimensional row vectors \( \left\{ b_1,\dots,b_m \right\}  \) and corresponding conical subdivision \( \Gamma=\left\{ \gamma_1, \dots, \gamma_m \right\} \). Let \(Z \sim \mathcal{N}(0,I_n) \),
\(
c_j =\Pr \left( Z \in \gamma_j \right),
\)
and \(u>0\) be such that
\[
 \Pr \left( | b_jZ-r| \le u \right) \ge 1-c_j\alpha
\]
 for \(j=1,\dots,m\), \(\alpha \in (0,1)\). Then \( \Pr \left( -u \le f(Z)-r \le u \right) \ge 1-\alpha\).
\end{lemma}
\begin{proof}
Let \(E_j\) be the event that \( \left\{ |b_jZ-r| \le u \mbox{ and } Z \in \gamma_j \right\} \). As argued previously \( \Pr \left( |f(Z)-r| \le u \right)= \sum\limits_{j=1}^m \Pr \left( E_j \right) \). Next note
\begin{eqnarray}
\Pr \left( E_j^c \right) && \le \Pr \left( Z \in \gamma_j^c \right) + \Pr \left( |b_jZ-r| > u \right) \nonumber\\
&& \le 1-c_j+ c_j\alpha = 1-(1-\alpha)c_j. \nonumber
\end{eqnarray}
Thus \( \Pr \left( E_j \right) \ge (1-\alpha)c_j \) and
\[
\Pr \left( |f(Z)-r| \le u \right)  = \sum\limits_{j=1}^m \Pr \left( E_j \right) \ge (1-\alpha)\sum\limits_{j=1}^m c_j = 1-\alpha.
\]
\end{proof} \qed

\begin{corollary}\label{t:bcor}
Let \(\alpha_j = \alpha \Pr\left( Z \in \gamma_j\right) \), then \(u_j = \|b_j\|\sqrt{\chi^2_1(\alpha_j)}\) will satisfy \( \Pr \left( |b_jZ| \le u_j \right) \allowbreak=1-\alpha_j, \) and \(u=\max\limits_{1\le j \le m}u_j\) satisfies \( \Pr \left( |f(Z)| \le u \right) \ge 1-\alpha\).
\end{corollary}

Note that while Corollary \ref{t:bcor} provides an upper bound for \(a^0(f)\), Lemma \ref{t:blem} can similarly be used to find upper bounds for \(a^r\) when \(r \ne0\). Additionally upper bounds for \(\eta^{\alpha}_j\) and  \(h^{\alpha}_j\) can be found as in Corollary \ref{t:bcor} by considering only the subset of cones \(\gamma_i\) indicated by \(z_N-z_N^*\).

\section{Numerical Examples}\label{s:examps}
This section applies the proposed methods and the method of Theorem \ref{t:indold} to two numerical examples. The half-width of intervals produced using the method of Theorem \ref{t:indold} will be denoted by \(\upsilon^{\alpha}_j\). When calculating \(a^r\), \(\eta^{\alpha}_j\) or \(h^{\alpha}_j\) for a function with three or more selection functions, the approach used throughout the examples is to perform a binary search with probabilities calculated as in \S \ref{s:comp} using the methods of \cite[Chapter 4]{g.b:np}. This search terminates when either the distance between the upper and lower bounds or the probability of the value being tested are within specified tolerance levels.

In each example we are able to find the true solution allowing us to examine the coverage rates for the different methods. For each example we generate 2,000 SAA problems at each sample size of \(N\)=50, 100, 200 and 2,000. For each sample the value of \(r\) used for \(a^r\) is chosen by generating i.i.d. \(Z_v \sim \mathcal{N}(0,I_{n})\), calculating
\[
r_N=10^{-3}\sum\limits_{v=1}^{10^3}\Phi_N^{-1}(z_N)\Sigma_N^{1/2}(Z_v),
\]
and taking the appropriate coordinate of this vector. The use of this procedure will be indicated with the notation \(a^{r_N}\).

\subsection*{Example 1}
For the first example, we consider a non-complementarity problem with
\[
S=\left\{x \in \Rset^2  \; \Big| \;  \left[\begin{array}{c c} .5 & -1 \\- 2 & 1 \end{array} \right] \left[ \begin{array}{c} x_1 \\ x_2 \end{array} \right]\le \left[ \begin{array}{c} 0\\ 0 \end{array} \right] \right\} \mbox{  and   }  F(x,\xi)=\left[ \begin{array}{c c}
 4  & 0 \\
3  & 2
\end{array} \right] \left[  \begin{array}{c} x_1 \\ x_2 \end{array}\right]  + \left[  \begin{array}{c} \xi_1 \\  \xi_2 \end{array}\right] ,
\]
where \( \xi\) is uniformly distributed over the box \([-1,1]\times[-2,2]\). In this case
\[
f_0(x)=\left[ \begin{array}{c c} 4 & 0\\ 3 & 2 \end{array} \right],
\]
and the SVI and its corresponding normal map formulation have true solutions \(x_0=z_0=0\). The function \(d(f_0)_{S}(z_0)\) is then piecewise linear, with the family of selection functions given by matrices
\[
\left[ \begin{array}{c c} 4 & 0\\ 3 & 2 \end{array} \right], \left[ \begin{array}{c c} 1.6 &  1.2 \\ 1 & 3 \end{array}\right],\left[ \begin{array}{c c} 1 & 0 \\ 0 & 1 \end{array}\right] \mbox{  and   } \left[ \begin{array}{c c} 3.4 & 1.2 \\ 2.8 & 2.4 \end{array}\right]
\]
and the corresponding conical subdivision \( \{K_1,\: K_2,\; K_3 \mbox{ and } K_4\}\) given by \(K_i=\allowbreak \{x\in \Rset^2  \big| C_ix \le 0 \} \) with
\[
C_1=\left[\begin{array}{c c} .5 & -1 \\- 2 & 1 \end{array} \right] \; C_2=\left[\begin{array}{c c} 2 & -1 \\- .5 & -1 \end{array} \right] \;
C_3=\left[\begin{array}{c c} .5 & 1 \\-2 & 1 \end{array} \right] \mbox{ and } C_4=\left[\begin{array}{c c}  -2& -1 \\- .5 & 1 \end{array} \right].
\]
 With this information we evaluate (\ref{q:inold}) for \(\alpha=.05\) and observe values of .9454 and .9461 for \(j=1\) and 2 respectively.

 In Tables \ref{tab:covex1z1} and \ref{tab:covex1z2} we summarize the coverage rates of  \( (z_0)_1\) and \( (z_0)_2\) for each interval determined by \(\upsilon^{\alpha}_j\), \(a^{r_N}\) and \(\eta^{\alpha}_j\).
\begin{table}[h]
\begin{minipage}[b]{0.45\linewidth}
\centering
\caption{Coverage rates \( (z_0)_1\) \(\alpha=.05\)}\label{tab:covex1z1}
\begin{tabular}{| l|c|c|c | }
\hline
& \(\upsilon^{\alpha}_1\) & \(a^{r_N}\) & \(\eta^{\alpha}_1\) \\ \hline
N=50 &  94.25\% &  94.75\%  &  94.2\% \\ \hline
N=100 &   94.55\% & 94.95 \%  &  94.9\% \\ \hline
N=200 &  94.1 \%  &  94.55 \% &   94.85 \% \\ \hline
N=2,000 &  94.7\% &  95.35\%  & 95.45\% \\ \hline
\end{tabular}
\end{minipage}
\hfill
\begin{minipage}[b]{0.45\linewidth}
\centering
\caption{Coverage rates \( (z_0)_2\) \(\alpha=.05\)}\label{tab:covex1z2}
\begin{tabular}{| l|c|c|c | }
\hline
& \(\upsilon^{\alpha}_2\) & \(a^{r_N}\) & \(\eta^{\alpha}_2\)  \\ \hline
N=50 &   93.8\% &   95.95\%  & 93.65 \%\\ \hline
N=100 &   94.15\% &  95.5\%  &  93.65\% \\ \hline
N=200 &    94.2\% &  95.25 \%  &  94.95\% \\ \hline
N=2,000 &   94.9\% &  95.45 \%  & 95.4\%\\ \hline
\end{tabular}
\end{minipage}
\end{table}
We see that the overall performance of the three approaches  is generally comparable and in line with the specified 95\% level of confidence and (\ref{q:inold}).

Differences between the methods become apparent in Figure \ref{fig:ex1n2k} where for the samples of size 2,000 we plot the length of intervals for \((z_0)_2\)  by which $K_i$ contains \( z_N-z_0\). These differences are further illustrated in Table \ref{tab:cbcex1n2k} where we break down the coverage of \( (z_0)_2\) and average interval length by which  $K_i$ contains  \(z_N-z_0\).

\begin{figure}[h!]
\centering
\subfloat[\(\upsilon^{\alpha}_2\)]{
\includegraphics{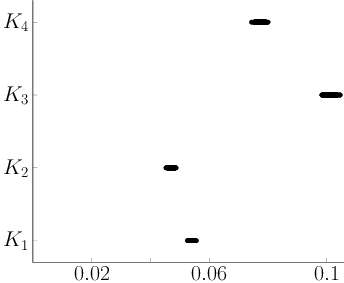}
}
\hfill
\subfloat[\(a^{r_N}\)]{
\includegraphics{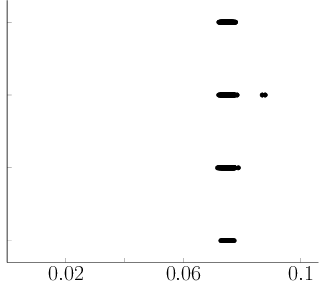}
}
\hfill
\subfloat[\(\eta^{\alpha}_2\)]{
\includegraphics{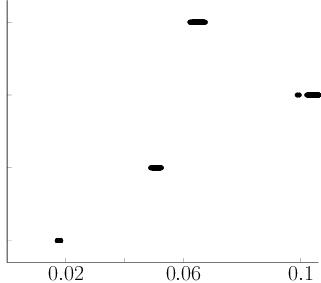}
}
\caption{Interval length for \( (z_0)_2\) by cone, \(N=2,000\)}\label{fig:ex1n2k}
\end{figure}
\begin{table}[h!]
\centering
\caption{Coverage of \( (z_0)_2 \) and half-width by cone, \(N=2,000\), \(\alpha=.05\)}\label{tab:cbcex1n2k}
\begin{tabular}{| c |c|c| c | c | c | c | c |  }
\hline
&  \multicolumn{3}{ | c |}{Coverage rate } &  \multicolumn{3}{ | c |}{Average length } \\ \cline{2-7}
Cone (samples in cone) &\(\upsilon^{\alpha}_2\) & \(a^{r_N}\) & \(\eta^{\alpha}_2\) & \(\upsilon^{\alpha}_2\) & \(a^{r_N}\) & \(\eta^{\alpha}_2\) \\ \hline
\(K_1\)(80) &  100\% &  100\% &  90\%&  .0541&  .0750&  .0177\\ \hline
\(K_2\) (689) & 92.31\% &  98.84\% &  95.21\%&  .0471&  .0749 & .0508\\ \hline
\(K_3\) (824) &  95.39\% &  90.29\% & 96.24 \%&  .1012& .0749 & .1051\\ \hline
\(K_4\) (407) & 97.3\% &  99.26\% &  95.09\%&   .0775&  .0749& .0649\\ \hline
\end{tabular}
\end{table}

The consistent value of \(a^{r_N}\) across samples is to be expected given Lemma \ref{t:lem1} and Proposition \ref{t:convprop}. Note that values of  \(a^{r_N}\) that deviate slightly from this pattern correspond to the two samples for which \(z_N^*\) was not contained in the relative interior of the same \(k\)-cell as \(z_0\). Across cones the performance of the intervals  varies, but this is accounted for in the definition of \(a^{r_N}\). Compare this with the intervals with half-width \(\upsilon^{\alpha}_2\). This approach does not directly account for the effect \(d(f_0)_S(z_0)\) being piecewise linear has on the asymptotic distribution of SAA solutions, and therefore the performance of the intervals. While in this example we can calculate (\ref{q:inold}) and observe that the intervals have an asymptotic level of confidence close to the desired 95\%, in general the varying performance across cones is not accounted for and the method may be unreliable. The value of \(\eta^{\alpha}_2\) also varies across cones, but its use of \(z_N-z_N^*\) and \(\Phi_N(z_N)\) to estimate the location of \(z_N-z_0\) in the conical subdivision of \(d(f_0)_S(z_0)\) allows for a level of confidence to be specified with less restrictive conditions. Additionally, the benefit of allowing \(\eta^{\alpha}_2\) to vary in a systematic way is seen in the more consistent performance of this approach across the four cones.

\begin{table}[h]
\centering
\caption{Coverage of \((x_0)_i\), \(\alpha=.05\)}\label{tab:x0cov}
\begin{tabular}{| l|c|c| }
\hline
& \((x_0)_1\) & \((x_0)_2\)  \\ \hline
N=50 &   96.05\% &   96.2\%  \\ \hline
N=100 &   97\% &   97.25\%  \\ \hline
N=200 &   97.1\% &   97.15\%  \\ \hline
N=2,000&   97\% &   95.33\%  \\ \hline
\end{tabular}
\end{table}

We next examine the performance of confidence intervals for \((x_0)_j\). For any real numbers \(l \le u\) neither \(\Pi_S(\Rset \times [l,u])\) nor \(\Pi_S( [l,u]) \times \Rset)\) result in sets that yield meaningful confidence intervals for \((x_0)_1\) or \((x_0)_2\). Therefore the indirect approach of projecting confidence intervals for \((z_0)_j\) onto \(S\) cannot be used and only the direct approach proposed in \S \ref{s:icim1} is applicable. Combining (\ref{q:icimres}) and the fact \(S \subset \Rset^2_+\) we consider \(\big[ \max\{0, (x_N)_j-N^{-1/2}h^{\alpha}_j\},\: (x_N)_j+N^{-1/2}h^{\alpha}_j \big]\) as the confidence interval for \((x_0)_j\).
\begin{table}[h!]
\centering
\caption{Intervals for \((x_0)_i\) by cone, \(N=2,000\), \(\alpha=.05\)}\label{tab:cex1n2k}
\begin{tabular}{|c| c | c | c | c |   }
\hline
&  \multicolumn{2}{ | c |}{Coverage rate } &  \multicolumn{2}{ | c |}{Average length } \\ \cline{2-5}
Cone (samples in cone) & \((x_0)_1\) & \((x_0)_2\) & \((x_0)_1\) & \((x_0)_2\)\\ \hline
\(K_1\)(80) &  88.75\% &   90\% &   .0104  & .0132 \\ \hline
\(K_2\) (689) &  95.36\% &   95.36\% &    .0089 & .0177 \\ \hline
\(K_3\) (824) &  100\% &   100\% &   0  &  0 \\ \hline
\(K_4\) (407) &  95.33\% &   95.33\% &   .0073  & .0036 \\ \hline
\end{tabular}
\end{table}
In Table \ref{tab:x0cov} we summarize the coverage of \((x_0)_1\) and \((x_0)_2\) at each sample size with \(\alpha=.05\) , and in Table \ref{tab:cex1n2k} we examine the performance and length of the intervals for the samples of size 2,000 broken down by the location of \(z_N-z_0\). Since the selection function corresponding to \(\Pi_{K_0}|_{K_3}\) is represented by the zero matrix when \(z_N-z_0 \in K_3\) the correct point \((x_N)_j=(x_0)_j=0\) is returned and as a result we see that the intervals for each component of \(x_0\) outperform the specified confidence level of \(95\%\).

\subsection*{Example 2}
For the second example we let \(S=\mathbb{R}^5_+\),
\[
F(x,\xi)=\left[ \begin{array}{c c c c c}
\xi_1 & 1.5 & .5 & .75 & .9 \\
1.5 & \xi_2  & 0 &.8 &1.5 \\
.5 & 0 & \xi_3 & .75 & 1.7 \\
.75 & .8 & .75 & \xi_4 & 1 \\
.9 & 1.5 & 1.7 & 1 & \xi_5
\end{array}\right] \left[ \begin{array}{c} x_1 \\ x_2 \\ x_3 \\ x_4 \\ x_5 \end{array}\right]
+ \left[\begin{array}{c} \xi_6 \\ \xi_7 \\ \xi_8 \\ \xi_9 \\ \xi_{10} \end{array}\right],
\]
 with \(\xi\) uniformly distributed over the box
\[
\left[2,4 \right] \times \left[0,4 \right] \times \left[0,3 \right] \times \left[2,6 \right] \times \left[-1,6 \right] \times \left[-1,1 \right] \times \left[-.5,.5 \right] \times \left[-2,2 \right] \times \left[-.75,.75 \right] \times \left[-1,1\right].
\]

The SVI and its normal map formulation have solutions \(x_0=z_0=0\). Moreover \(\Pi_{\mathbb{R}^5_+}=d\Pi_{\mathbb{R}^5_+}(z_0)\) with
\[
d\Pi_{\mathbb{R}^5_+}(z_0)(x)= \left[ \begin{array}{c c c c c}
h_1 & 0 & 0 & 0 & 0 \\
0 & h_2 & 0 & 0 & 0 \\
0 & 0 & h_3 & 0 & 0 \\
0 & 0 & 0 & h_4 & 0 \\
0 & 0 & 0  & 0 & h_5 \end{array}\right] \left[\begin{array}{c} x_1 \\ x_2 \\ x_3 \\ x_4 \\ x_5 \end{array}\right] \;  \mbox{ where } h_i =
\left\{
\begin{array}{lll}
0 & \mbox{if} & x_i\le 0, \\
1 & \mbox{if} & x_i \ge 0,
\end{array} \right.
\]
so \(d(f_0)_{\mathbb{R}^5_+}(z_0)(\cdot)\) is piecewise linear with a family of thirty-two selection functions. Taking \( \alpha=.05\) we first consider confidence intervals for \((z_0)_j\). Evaluating (\ref{q:inold}) for each value of \(j=1,\dots,5\) we observe that the intervals for \((z_0)_j\) considered in Theorem \ref{t:indold} have asymptotic levels of confidence  of 93.85\%, 93.33\%, 94.38\%, 93.39\% and 92.96\% respectively.

\begin{table}[h!]
\centering
\caption{Coverage rates for \( (z_0)_3\)}\label{tab:covz4}
\begin{tabular}{|c | c | c| c|}
\hline
 & \(\upsilon^{\alpha}_4\) & \(a^{r_N}\) & \(\eta^{\alpha}_4\) \\ \hline
\(N=50\) & 93.05\% & 96.3 \% & 93.3\%\\ \hline
\(N=100\) & 92.85\% & 99.95 \% & 92.8\% \\ \hline
\(N=200\) & 94\% & 94.7 \% & 94.95\%\\ \hline
\(N=2,000\) & 94.35\% & 94.6 \% & 94.8\%\\ \hline
\end{tabular}
\end{table}

Coverage rates of the confidence intervals are largely in line with the specified level of confidence or as indicated by (\ref{q:inold}), with the coverage rates of \( (z_0)_3\) summarized in Table \ref{tab:covz4} for each approach and sample size considered. The performance of the different methods broken down by where \(z_N-z_0\) falls in the conical subdivision associated with \(d(f_0)_S(z_0)\) cannot be as well observed, given the large number of cones relative to the number of samples. What we are able to observe is the consistent values of \(a^{r_N}\) across samples as compared to the values of \(\upsilon^{\alpha}_j\) and \(\eta^{\alpha}_j\), shown in Figure \ref{fig:exsz4} for \( (z_0)_3\) and \(N=\)2,000.
\begin{figure}[h]
\centering
\includegraphics[scale=.7]{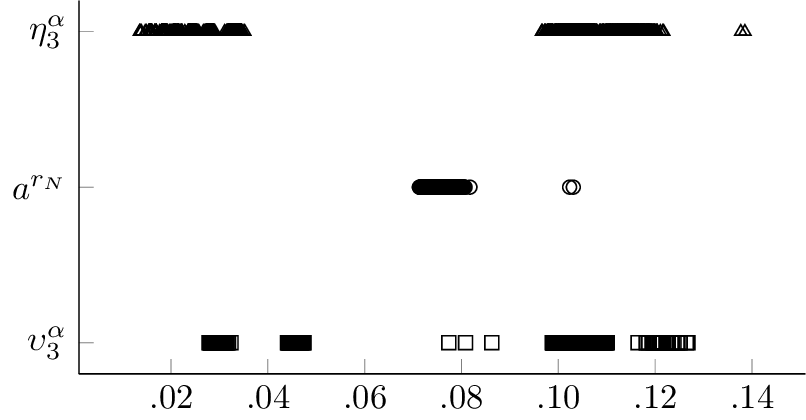}
\caption{Intervals Lengths for \( (z_0)_3, N=2,000\)}\label{fig:exsz4}
\end{figure}
 Note there are two samples for which the value of \(a^{r_N}\) deviate from this pattern, and as in the previous example they correspond to samples for which \(z_N^*\) and \(z_0\) are not contained in the relative interior of the same \(k\)-cell. In this example the computational benefits of \(\eta^{\alpha}_j\) are clear. For almost all of the samples calculating \(a^{r_N}\) required working with a piecewise linear function with thirty two selection functions, whereas for all of the samples calculating \(\eta^{\alpha}_j\) only involved a single selection function, leading to a dramatic reduction in the necessary computation.

With this example we also examine how upper bounds satisfying the conditions of Lemma \ref{t:blem} compare to the actual half-widths. In Table \ref{tab:btab} we summarize average and median ratio of bound to actual half-width for samples of size \(N=2,000\). While easier to compute we see that the bounds can be quite conservative. This is in large part due to their dependance on estimates of \(\Pr \left( Z \in T_i \right)\), especially in the case of bounds for \(a^{r_N}\) which require considering each \(T_i\).

\begin{table}[h]
\centering
\caption{Ratio of upper bound to interval half-width}\label{tab:btab}
\begin{tabular}{| l c | c | c | c | c| }
\hline
& & \multicolumn{2}{c |}{ \(a^{r_N}\)  } & \multicolumn{2}{c |}{ \(\eta^{\alpha}_j\) } \\ \cline{3-6}
 &   & Average ratio  & Median ratio  & Average ratio  & Median ratio \\ \hline
\multirow{5}{*}{\(N=2,000\)} & \((z_0)_1\) & 6.20&  6.33& 3.04 & 2.18 \\
 & \((z_0)_2\) &15.53 & 13.44 & 3.58 &  2.92 \\
 & \((z_0)_3\) &4.00 &  3.49 & 2.25  & 1.55 \\
 & \((z_0)_4\)&  5.27 & 5.26 &  3.69 & 2.37 \\
 & \((z_0)_5\) &  9.20 &  8.04 & 2.80 & 2.12 \\  \hline
\end{tabular}
\end{table}

When computing intervals for \((x_0)_j\)  note that since \(S=\Rset_+^5\) each selection function of \(d\Pi_S(z_N^*)\) is represented by a diagonal matrix with values of zero and one along the diagonal. When \(z_N-z_N^*\) falls into a cone for which the \(j^{\tiny\mbox{th}}\) diagonal element of the selection function's matrix representation is one from (\ref{q:adef2}) and (\ref{q:hici}) we see that \(\eta^{\alpha}_j\) and \(h^{\alpha}_j\) will be equal. The interval for \((x_0)_j\) produced using the approach of \S \ref{s:icim1} would then be the same as the projection onto \(S\) of the interval for \((z_0)_j\) produced using the approach of \S \ref{s:nmicim2}. If \(j^{\tiny\mbox{th}}\) diagonal element is zero the method of \S \ref{s:icim1} returns the correct point estimate \((x_N)_j=0\), whereas the projection onto \(S\) of the interval for \((z_0)_j\) produced using the approach of \S \ref{s:nmicim2} is given by \( \big[0,\max\big\{0,\:(z_N)_j+N^{-1/2}\eta^{\alpha}_j\big\} \big]\).
\begin{figure}[h]
\centering
\includegraphics[scale=.9]{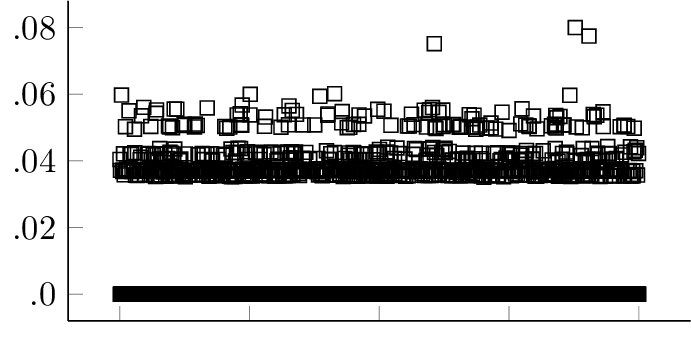}
\caption{Bounds for \( h^{\alpha}_3, N=2,000\) by sample}\label{fig:dbsx3}
\end{figure}
 The two approaches therefore produce intervals that cover \( (x_0)_j=0\) at an identical rate with the approach of \S \ref{s:icim1} returning the correct point estimate more often. Moreover the bound for \(h^{\alpha}_j\) provided by Corollary \ref{t:bcor} will have similar properties since these bounds consider adjusting only the value of \(\alpha\) and not the selection functions used. Therefore when the  \(j^{\tiny\mbox{th}}\) diagonal element of the indicated selection function is one the bounds for \(h^{\alpha}_j\) and \(\eta^{\alpha}_j\) will be the same, and if the \(j^{\tiny\mbox{th}}\) diagonal element is zero the bound for \(h^{\alpha}_j\) is also zero. This is illustrated in Figure \ref{fig:dbsx3} where we have plotted the bounds for \(h^{\alpha}_3\) for each sample of size 2,000.

\begin{acknowledgements}
Research of Michael Lamm and Shu Lu is supported by National Science Foundation under the grant DMS
-1109099. Research of Amarjit Budhiraja is supported in part by the National Science Foundation (DMS-1004418,DMS-1016441, DMS-1305120)
and the Army Research Office (W911NF-10-1-0158).
\end{acknowledgements}

\bibliographystyle{spmpsci}
\bibliography{svi}

\end{document}